\documentclass[11pt,a4paper,reqno]{amsart}
\usepackage{amsfonts}
\usepackage{amsmath,latexsym,amssymb,amsfonts,amsbsy, amsthm}
\usepackage{a4wide}
\usepackage[usenames]{color}
\usepackage{xspace,colortbl}
\allowdisplaybreaks
\usepackage{color}
\usepackage[hidelinks]{hyperref} 
\newcommand{\beq}{\begin{equation}}
\newcommand{\eeq}{\end{equation}}
\newcommand{\ben}{\begin{eqnarray}}
\newcommand{\een}{\end{eqnarray}}
\newcommand{\beno}{\begin{eqnarray*}}
\newcommand{\eeno}{\end{eqnarray*}}

\theoremstyle{plain}
\newtheorem{theorem}{Theorem}[section]

\newtheorem{corollary}[theorem]{Corollary}
\newtheorem{proposition}[theorem]{Proposition}
\newtheorem{lemma}[theorem]{Lemma}
\newtheorem{remark}[theorem]{Remark}

\numberwithin{theorem}{section} \numberwithin{equation}{section}

\newcommand{\average}{{\mathchoice {\kern1ex\vcenter{\hrule height.4pt
width 6pt depth0pt} \kern-9.7pt} {\kern1ex\vcenter{\hrule height.4pt
width 4.3pt depth0pt} \kern-7pt} {} {} }}

\def\R{\mathbb{R}}

\newcommand{\del}{\partial }

\renewcommand{\phi}{\varphi}

\newcommand{\be}{\begin{equation}}
\newcommand{\ee}{\end{equation}}

\newcommand{\N}{\mathbb{N}}

\newcommand{\cM}{{\mathcal M}}
\newcommand{\cN}{{\mathcal N}}

\newcommand{\supp}{{\rm supp}}
\newcommand{\weakto}{\rightharpoonup}
\newcommand{\weak}{\rightharpoonup}

\newcommand{\embed}{\hookrightarrow}

\newcommand{\eps}{\varepsilon}

\renewcommand{\epsilon}{\varepsilon}

\title{Spiraling solutions of nonlinear Schr\"odinger equations}

\author[Oscar Agudelo]
{Oscar Agudelo}
\address{NTIS - New Technologies for the Information Society,
Faculty of Applied Sciences,
Technicka 8,
Pilsen, 301 00,
Czech Republic}
\email{oiagudel@ntis.zcu.cz}

\author[Joel K\"ubler]
{Joel K\"ubler}
\address{Institut f\"ur Mathematik,
Goethe-Universit\"at Frankfurt,
Robert-Mayer-Str. 10,
D-60629 Frankfurt am Main, Germany}
\email{kuebler@math.uni-frankfurt.de}

\author[Tobias Weth]
{Tobias Weth}
\address{Institut f\"ur Mathematik,
Goethe-Universit\"at Frankfurt,
Robert-Mayer-Str. 10,
D-60629 Frankfurt am Main, Germany}
\email{weth@math.uni-frankfurt.de}

\begin{document}

\maketitle

\begin{abstract}
  We study a new family of sign-changing solutions to the stationary nonlinear Schr{\"o}dinger equation
$$
-\Delta v +q v =|v|^{p-2} v, \qquad \text{in $\R^3$,}
$$
with $2<p<\infty$ and $q \ge 0$. These solutions are spiraling in the sense that they are not axially symmetric but invariant under screw motion, i.e., they share the symmetry properties of a helicoid. In addition to existence results, we provide
information on the shape of spiraling solutions, which depends on the parameter value representing the rotational slope of the underlying screw motion. Our results complement
a related analysis of Del Pino, Musso and Pacard in \cite{del Pino-Musso-Pacard} for the Allen-Cahn equation, whereas the nature of results and the underlying variational structure are completely different.   
\end{abstract}

\section{Introduction}
The present paper is concerned with a new class of solutions to the stationary nonlinear Schr{\"o}dinger equation 
\begin{equation}
  \label{eq:main-eq-R-N}
-\Delta v +q v =|v|^{p-2} v \qquad \text{in $\R^N$,}  
\end{equation}
where $p>2$ and $q \ge 0$ is a constant. Since the case $q>0$ is equivalent to $q=1$ by rescaling, we only consider the cases $q=1$ and $q=0$ in the following.

For subcritical exponents $p$ (i.e., $p< \frac{2N}{N-2}$ if $N \ge 3$) and $q=1$, there is a vast literature on solutions of (\ref{eq:main-eq-R-N}) in $H^1(\R^N)$,  which decay expontially at infinity, see e.g. the monographs \cite{ambrosetti-malchiodi-a,kuzin-pohozaev,Struwe,sulem-sulem,Willem} and the references therein. In the present paper, we focus on solutions with only partial decay. These solutions are less understood, but have attracted considerable attention in recent years. 

To be more precise, let us write $\bar x=(x,t) \in \R^N$ with $x \in \R^{N-1}$ and $t\in \R$. We shall consider solutions $v: \R^N \to \R$ satisfying 
\begin{equation}
  \label{eq:decay-property}
\lim \limits_{|x| \to \infty}v(x,t) =0 \quad \text{uniformly in } t.  
\end{equation}
A trivial class of solutions satisfying \eqref{eq:decay-property} is the class of solutions that are axially symmetric with 
respect to the axis  $\{(0_{\R^{N-1}},t): \ t \in \R \} \subset \R^N$ 
and that in addition are $t$-invariant, i.e., solutions having the form $v(x,t)= \tilde v(x)$, where $\tilde v$ is a radial solution of (\ref{eq:main-eq-R-N}) in $\R^{N-1}$ satisfying $\tilde v(x) \to 0$ as $|x| \to \infty$. Here and in the following, axial symmetry is always understood with respect to the $t$-axis.

In a seminal paper, Dancer \cite{dancer} constructed, for $q=1$, nontrivial, $t$-periodic axially symmetric solutions of (\ref{eq:main-eq-R-N}) by means of bifurcation theory. The solutions found in \cite{dancer} are positive, and they bifurcate from the unique family of $t$-invariant axially symmetric positive solutions of (\ref{eq:main-eq-R-N}). 

It is natural to ask whether, for a given {\em positive} solution of (\ref{eq:main-eq-R-N}), the decay property (\ref{eq:decay-property}) enforces axial symmetry up to translations. As shown in the following theorem by Farina, Malchiodi and Rizzi in \cite{farina-malchiodi-rizzi}, this is true for positive solutions which are periodic in $t$.  
\begin{theorem} \cite[Special case of Theorem 2]{farina-malchiodi-rizzi}\\ 
\label{symmetry of positive solutions}
Let $p>2$, $q=1$, and let $v \in C^2(\R^N)$ be a bounded positive solution of (\ref{eq:main-eq-R-N}) satisfying the uniform decay property (\ref{eq:decay-property}). Suppose moreover that $v$ is periodic in $t$, i.e., there exists $\tau >0$ with 
$$
v(x,t+ \tau)= v(x,t) \qquad \text{for all $(x,t) \in \R^N$.}
$$ 
Then, up to translations in the $x$-variable, $v$ is axially symmetric.
\end{theorem}

Let us also briefly discuss the case $q=0$ in (\ref{eq:main-eq-R-N}). In this case, for subcritical $p$, it is known that (\ref{eq:main-eq-R-N}) does not admit positive solutions (see \cite[Theorem 1.1]{gidas-spruck}), and it also does not admit solutions of any sign in $H^1(\R^N)$ (by Pohozaev's identity, see e.g. \cite[Appendix B]{Willem}). The latter property is related to the fact that, in this case, equation (\ref{eq:main-eq-R-N}) remains invariant under the rescaling transformation $v \mapsto \kappa^{\frac{2}{p-2}}v(\kappa\,\cdot\,)$. 
\medskip

In the present paper, we discuss solutions of (\ref{eq:main-eq-R-N}) - (\ref{eq:decay-property}) with periodicity in $t$, but without axial symmetry. By Theorem~\ref{symmetry of positive solutions} and the remarks above, such solutions have to change sign.
As far as we know, solutions of this type have not been studied yet with the exception of the trivial $t$-independent case where $v(x,t)= \tilde v(x)$ for some non-radial sign-changing solution $\tilde v$ of (\ref{eq:main-eq-R-N}) in $\R^{N-1}$.

We restrict our attention to the case $N =3$ and consider the special class of {\em spiraling solutions} of the nonlinear Schr\"odinger equation 
\begin{equation}
  \label{eq:main-eq}
-\Delta v +q v =|v|^{p-2} v \qquad \text{in } \R^3,  
\end{equation}
i.e., solutions that are invariant under the action of a screw motion.

\medskip
To be more precise, let $\lambda>0$. We call a function $v: \R^3 \to \R$ $\lambda-$spiraling  if for any $\theta \in \R$,
\begin{equation}
  \label{eq:invariance-condition}
v(R_{\theta} x,t + \lambda \theta)= v(x,t) \qquad \text{for $x \in \R^2$, $t\in \R$,}
\end{equation}
where $R_\theta: \R^2 \to \R^2$ denotes the counter-clockwise rotation with angle $\theta$ in $\R^2$. Notice that $\lambda-$spiraling functions are $2 \lambda \pi$-periodic in $t$. Hence, the parameter $\lambda$ represents the rotational slope of the underlying screw motion, and $2 \lambda \pi$ is the associated turn-around shift. 

\medskip
Our work is partly inspired by the papers \cite{del Pino-Musso-Pacard} resp. \cite{Cinti-Davila-Del Pino} where spiraling solutions have been constructed for the classical and fractional Allen-Cahn equation, respectively. Without going into detail, we mention the well known fact that, despite its similar looking form,  the Allen-Cahn equation $-\Delta u = u - u^3$ differs significantly from the nonlinear Schr{\"o}dinger equation~(\ref{eq:main-eq}) with regard to the variational framework and the shape of solutions.

\medskip
In cylindrical coordinates $(x,t)= (r \cos \varphi, r \sin \varphi, t)$ with $(r, \varphi,t) \in [0,\infty) \times \R \times \R$, $\lambda-$spiraling functions have the form 
$$
v(r,\varphi,  t) = u \left(r,\varphi- \frac{t}{\lambda} \right)
$$
with a function 
$u: [0,\infty) \times \R \to \R$ which is $2 \pi$-periodic in the second variable. Also, in these coordinates the equation (\ref{eq:main-eq}) reads as 
$$
-v_{rr}-\frac{v_r}{r} - \frac{v_{\varphi \varphi }}{r^2} -  v_{tt} +q\,v = |v|^{p-2} v
$$
so that the equation for $u$ has the form
\begin{equation}
  \label{eq:v-equation-polar}
-u_{rr} -\frac{u_r}{r} - \Bigl(\frac{1}{\lambda^2}+ \frac{1}{r^2} \Bigr)u_{\theta\theta} + q\,u= |u|^{p-2} u .
\end{equation}

It is convenient to transform equation \eqref{eq:v-equation-polar} further to planar euclidean coordinates $x=(x_1,x_2)$, where $r= |x|$ and $\theta= \arcsin \frac{x_2}{|x|}$. This leads to the problem
\begin{equation}
\label{eq:untransformed-problem whole space}
\left \{
\begin{aligned}
-\Delta u - \frac{1}{\lambda^2} [x_1 \partial_{x_2} - x_2 \partial_{x_1}]^2 u + q\, u &= |u|^{p-2} u &&\qquad \text{on \,\,$\R^2$, }\\
u(x) & \to 0 &&   \qquad \text{as \,\,$|x| \to \infty$.}
\end{aligned}
\right.  
\end{equation}

Observe that radial solutions of (\ref{eq:untransformed-problem whole space}) correspond to axially symmetric and $t$-invariant solutions of (\ref{eq:main-eq}). By Theorem~\ref{symmetry of positive solutions}, every positive solution of (\ref{eq:untransformed-problem whole space}) is radial. On the other hand, nonradial solutions of (\ref{eq:untransformed-problem whole space}) correspond to solutions of (\ref{eq:main-eq}) which are $2 \lambda \pi$-periodic in $t$ but neither axially symmetric nor $t$-invariant. We therefore restrict our attention to nodal (i.e., sign-changing) solutions of \eqref{eq:untransformed-problem whole space}.

\medskip
We study problem \eqref{eq:untransformed-problem whole space} using variational  methods, and hence we first introduce some notation related to its variational structure. 

\medskip
We write $\del_\theta:=  x_1 \del_{x_2} - x_2 \del_{x_1}$ for the angular derivative and consider the space 
\begin{equation}
  \label{eq:def-H}
H:= \left\{u \in H^1(\R^2)\::\: \int_{\R^2} |\del_\theta u|^2 dx <\infty \right\}.
\end{equation}
For $\lambda>0$, we endow $H$ with the $\lambda$-dependent scalar product
\begin{equation}\label{eq:lambda-product}
\langle u,v \rangle_{\lambda} := \int_{\R^2} \Bigl(\nabla u \cdot \nabla v + \frac{1}{\lambda^2}
(\del_\theta u)(\del_\theta v) +  \,uv \Bigr) dx
\end{equation}
and consider the Hilbert space $(H,\langle\cdot,\cdot\rangle_{\lambda})$. 

\medskip
Let $E_\lambda: H \to \R$ be the energy functional associated to \eqref{eq:untransformed-problem whole space} in the case $q=1$, defined by 
\begin{equation}
  \label{eq:def-energy-functional-intro}
E_\lambda(u) := \frac{1}{2} \int_{\R^2} \Bigl(|\nabla u|^2 + \frac{1}{\lambda^2}
|\del_\theta u|^2 + |u|^2 \Bigr) dx -
\frac{1}{p}\int_{\R^2} |u|^p \,dx.
\end{equation}
By standard arguments, $E_\lambda$ is of class $C^1$, and critical points of $E_\lambda$ are weak solutions of (\ref{eq:untransformed-problem whole space}). 

\medskip
By definition, a least energy nodal solution of \eqref{eq:untransformed-problem whole space} is a minimizer of $E_\lambda$ within the class of sign-changing solutions of (\ref{eq:untransformed-problem whole space}). Our first main result is concerned with least energy nodal solutions and reads as follows.

\begin{theorem} \label{main theorem}
	Let $p > 2$ and $q=1$. 
	For every $\lambda>0$ there exists a least energy nodal solution of \eqref{eq:untransformed-problem whole space}. Furthermore, there exist $0<\lambda_0 \leq \Lambda_0 < \infty$ with the following properties:
	\begin{itemize}
		\item[(i)]
                  For $\lambda < \lambda_0$, every least energy nodal solution of \eqref{eq:untransformed-problem whole space} is radial.
                  \smallskip
                  
		\item[(ii)]
		For $\lambda > \Lambda_0$, every least energy nodal solution of \eqref{eq:untransformed-problem whole space} is nonradial.
	\end{itemize}
\end{theorem}

Theorem~\ref{main theorem} establishes a symmetry breaking phenomenon for least energy nodal solutions which occurs within a finite range of parameters $\lambda \in [\lambda_0,\Lambda_0]$. We are not aware of any other setting where such a transition from radiality to nonradiality has been observed for least energy nodal solutions. The main difficulty when dealing with least energy radial nodal solutions of the equation $-\Delta u + u = |u|^{p-2}u$ in $\R^2$ is given by the fact that so far neither uniqueness (up to sign) nor nondegeneracy is known. Hence, in order to prove the first part of Theorem~\ref{main theorem}, we have to follow an approach which does not rely on these properties. In fact,
a more general radiality result for solutions of (\ref{eq:untransformed-problem whole space}) with small $\lambda>0$ can be obtained by combining uniform elliptic $L^\infty$-estimates with Poincar\'e type inequalities in the angular variable. More precisely, we have the following.

\begin{theorem} \label{non-radiality-L-infty-est-intro}
  Let $p >2$ and $q=1$.
  \begin{enumerate}
  \item If $u \in H$ is a nontrivial weak solution of (\ref{eq:untransformed-problem whole space}) for some $\lambda>0$ satisfying $\lambda < \Bigl(\frac{1}{(p-1)|u|_\infty^{p-2}}\Bigr)^{\frac{1}{2}}$, then $u$ is a radial function. 
  \item For every $c>0$, there exists $\lambda_c>0$ with the property that every weak solution $u \in H$ of (\ref{eq:untransformed-problem whole space})
for some $\lambda \in (0,\lambda_c)$ with $E_\lambda(u) \le c$ is radial.
\end{enumerate}
\end{theorem}

The first part of Theorem~\ref{main theorem} turns out to be a consequence of Theorem~\ref{non-radiality-L-infty-est-intro}(ii) and uniform (in $\lambda$) energy estimates for least energy nodal solutions of (\ref{eq:untransformed-problem whole space}) in the case $p>2$, $q=1$, see Section~\ref{sec:radi-vers-nonr} below.

\medskip
While least energy nodal solutions are particularly interesting from a variational point of view, Theorem~\ref{main theorem}(i) and Theorem~\ref{non-radiality-L-infty-est-intro}(ii) show that,  in order to detect nonradial sign-changing  solutions of (\ref{eq:untransformed-problem whole space}) for small values $\lambda>0$,  we have to pass to higher energy levels. A natural class of nonradial nodal solutions of (\ref{eq:untransformed-problem whole space}) is the class of odd solutions with respect to a hyperplane reflection. 

\medskip
If we consider the hyperplane $\{x_1=0\}$, then any such solution corresponds to a solution of the boundary value problem 
\begin{equation}
  \label{eq:untransformed-problem}
  \left \{
    \begin{aligned}
-\Delta u - \frac{1}{\lambda^2} [x_1 \partial_{x_2} - x_2 \partial_{x_1}]^2 u + q\,u &= |u|^{p-2}u &&\qquad \text{on $\R^2_+$,}\\
u &=0 && \qquad \text{on $\partial \R^2_+$}
    \end{aligned}
\right.  
\end{equation}
in the half space $\R^2_+:= \{x \in \R^2\::\: x_1>0\}$. Moreover, by odd reflection and transformation of coordinates, any such solution $u$ gives rise to a $\lambda-$spiraling nodal solution $v: \R^3 \to \R$ of (\ref{eq:main-eq}) with the property that  
$$
v(0,t)=0= v(R_{t}(0,x_2), \lambda t) \qquad \text{for all $t, x_2 \in \R$.}
$$
Consequently, $v$ vanishes on a helicoid, i.e. the condition $u =0$ on $\del \R_+^2$ implies that $v$ is zero on the set $\left\{ (x \sin t, x \cos t, \lambda t): t,x \in \R \right\}$.
\medskip

Weak solutions of (\ref{eq:untransformed-problem}) correspond to critical points of the $C^1$-functional $E_{\lambda}^+: H^+ \to \R$ defined by
\begin{equation}
  \label{eq:def-E-plus}
E_{\lambda}^+(u) := \frac{1}{2} \int_{\R^2_+} \bigl(|\nabla u|^2 + \frac{1}{\lambda^2} |\del_\theta u|^2 + q u^2 \bigr)dx- \frac{1}{p}
\int_{\R^2_+} |u|^p \,dx ,
\end{equation}
where \begin{equation}
  \label{eq:def-H-plus}
H^+:= \left\{u \in H^1_0(\R^2_+)\::\: \int_{\R^2_+} |\del_\theta u|^2 dx <\infty \right\}.
\end{equation}
By trivial extension, we regard $H^+$ as a closed subspace of $H$, see Section~\ref{sec:exist-symm-odd} below for details.

\medskip
Our main result for (\ref{eq:untransformed-problem}) reads as follows. 

\begin{theorem}
\label{exist-mountain-pass}  
Let $p>2$, $q \in \{0,1\}$ and $\lambda>0$. 

\begin{enumerate}
\item[(i)] (Existence) Problem (\ref{eq:untransformed-problem}) admits a positive least energy solution. 
\item[(ii)] (Symmetry) Any positive solution $u$ of (\ref{eq:untransformed-problem}) is symmetric with respect to reflection at the $x_1$-axis and decreasing in the angle $|\theta|$ from the $x_1$-axis. In particular, $u$ takes its maximum on the $x_1$-axis. 
\item[(iii)] (Asymptotics) If $q=1$ and $\lambda_k \ge 1$ are given with $\lambda_k \to + \infty$ as $k \to \infty$ and $u_k$ is a positive least energy solution of (\ref{eq:untransformed-problem}) with $\lambda= \lambda_k$, then, after 
passing to a subsequence,  there exists a sequence of numbers $\tau_k >0$ with 
$$
\tau_k \to + \infty, \qquad  \frac{\tau_k}{\lambda_k} \to 0 \qquad \text{as $k \to \infty$}
$$
such that the translated functions $w_k \in H^1(\R^2)$, $w_k(x)= u_k(x_1+ \tau_k,x_2)$ satisfy
$$
w_k \to w_\infty \qquad \text{strongly in $H^1(\R^2)$,}
$$
where $w_\infty$ is the unique positive radial solution of 
\begin{equation}
  \label{eq:unique-rad-pos-solution-limit-problem}
-\Delta w_\infty + w_\infty = |w_\infty|^{p-2}w_\infty, \qquad w_\infty \in H^1(\R^2).
\end{equation}
\end{enumerate}
\end{theorem}
Similarly as defined for the equation \eqref{eq:untransformed-problem whole space}, a least energy solution of \eqref{eq:untransformed-problem} is, by definition, an energy minimizer within the class of nontrivial solutions of (\ref{eq:untransformed-problem}). More specifically, least energy solutions will be characterized as minimizers of $E_{\lambda}^+$ w.r.t. the associated Nehari manifold and attain the mountain pass level
\begin{equation}\label{eq:MP-level}
c_{\lambda}=\inf_{u \in H^+ \setminus \{ 0\}} \sup_{t \geq 0} E_{\lambda}^+(tu),
\end{equation}
see Section~\ref{sec:exist-symm-odd} below. We also point out that the uniqueness of a positive radial solution to \eqref{eq:unique-rad-pos-solution-limit-problem} was shown by Kwong \cite{Kwong}. 

\begin{remark}
\label{remark-intro}{\rm
\begin{itemize} 
	\item[(i)] Let $p>2$ and $q=1$. 
As a consequence of Theorem~\ref{exist-mountain-pass}, the energy of the least energy nodal solution of (\ref{eq:untransformed-problem whole space}), as considered in Theorem \ref{main theorem}, tends to $2 c_\infty$ as $\lambda \to \infty$, where $c_\infty$ is the least energy of nontrivial solutions of the limit problem (\ref{eq:unique-rad-pos-solution-limit-problem}). This fact is the key ingredient in the proof of Theorem~\ref{main theorem}(ii).

\medskip
\item[(ii)] The existence result for (\ref{eq:untransformed-problem}) for $p>2$ and $q \in \{0,1\}$ relies on compact embeddings. More precisely, we will prove in Section~\ref{sec:funct-analyt-fram} below that the space $H$ is compactly embedded into $L^\rho(\R^2)$ for $\rho \in (2,\infty)$, which readily implies that the space $H^+$ is compactly embedded in $L^\rho(\R^2_+)$ for $\rho \in (2,\infty)$. With the help of these embeddings and by applying the symmetric mountain pass theorem (see Theorem 6.5 in \cite{Struwe}), we may also prove, for any $\lambda>0$, the existence of a sequence of pairs of solutions $\pm u_j$ whose sequence of energies is unbounded.
\end{itemize}}
\end{remark}

The existence and symmetry parts of Theorem~\ref{exist-mountain-pass} extend to a larger class of semilinear equations, see Section~\ref{sec:exist-symm-odd} below. Next, we shall see that the case $q=0$ in (\ref{eq:untransformed-problem}) arises naturally when considering the asymptotics of positive least energy solutions of (\ref{eq:untransformed-problem}) in the case $q=1$ when $\lambda \to 0$.
We shall see that these solutions concentrate at the origin as $\lambda \to 0$. More precisely, we have the following. 

\begin{theorem}  \label{rescaled convergence}
Let $(\lambda_k)_k$ be sequence of numbers $\lambda_k \le 1$ such that $\lambda_k \to 0$ as $k \to \infty$. Moreover, let $u_k \in H^+$ be a positive least energy solution
of \eqref{eq:untransformed-problem} with $q=1$, and let $v_k \in H^+$ be defined by $v_k(x)= \lambda_k^{\frac{2}{p-2}}u_k(\lambda_k x).$
\medskip

Then, after passing to a subsequence, we have $ v_k  \to v^*$ in $H^+$, where $v^*$ is a positive least energy solution of the problem 
\begin{equation}
  \label{eq:untransformed-problem-limit-lambda-0}
  \left \{
    \begin{aligned}
-\Delta v^* - [x_1 \partial_{x_2} - x_2 \partial_{x_1}]^2 v^*&= |v^*|^{p-2}v^* &&\qquad \text{on $\R^2_+$,}\\
v^* &=0 && \qquad \text{on $\partial \R^2_+$}
    \end{aligned}
\right.  
\end{equation}
\end{theorem}

\begin{remark}{\rm
The statements given in Theorems~\ref{exist-mountain-pass}(i) and \ref{rescaled convergence} remain valid when the underlying half space $\R^2_+$ is replaced by the cone 
$$
C_{\alpha}:= \{x \in \R^2\::\: x_1>0, \quad \arcsin \frac{x_2}{|x|} < \alpha\}.
$$
In particular, in the case where $ \alpha = \frac{\pi}{2 j}$ for some positive integer $j$, successive reflection yields solutions with precisely $2j$ nodal domains.}
\end{remark}

\medskip
The paper is organized as follows. Section~\ref{sec:funct-analyt-fram} sets up the functional analytic framework and provides some preliminary results. In particular, we shall prove the compactness of the embedding $H \hookrightarrow L^\rho(\R^2)$ for $\rho \in (2,\infty)$, and we establish the existence of least energy nodal solutions for problem \eqref{eq:untransformed-problem whole space}. In Section~\ref{sec:exist-symm-odd}, we study the symmetry and existence of ground state solutions for a generalization of problem \eqref{eq:untransformed-problem}. In Section~\ref{sec:asympt-via-energy-lemma} we discuss the asymptotics of least energy solutions to (\ref{eq:untransformed-problem}) as $\lambda \to \infty$ and as $\lambda \to 0$ and prove Theorems \ref{exist-mountain-pass} and \ref{rescaled convergence}. 
Finally, Section~\ref{sec:radi-vers-nonr} is devoted to the proofs of Theorem \ref{main theorem} and Theorem~\ref{non-radiality-L-infty-est-intro}.
In the appendix, we prove a result on uniform $L^\infty$-bounds for weak solutions of \eqref{eq:untransformed-problem whole space} in the case $q=1$.

\section{Preliminary results}
\label{sec:funct-analyt-fram}
In the following, all functions are assumed to be real-valued. We consider the space $H$ defined in (\ref{eq:def-H}) with the $\lambda$-dependent scalar product defined in \eqref{eq:lambda-product} with $\|\cdot\|_\lambda$ denoting the corresponding norm. The space $(H, \langle \cdot, \cdot \rangle_\lambda)$ is a Hilbert space and clearly, all the norms $\|\cdot\|_\lambda$, $\lambda>0$, are equivalent.

\medskip
For easier distinction from the norms on $H$, for $\rho \in [1,\infty]$, we will use the notation $|\cdot|_{\rho}$ to denote the standard norm on $L^{\rho}(\R^2)$.

\medskip
Recall also that we have set $\del_\theta:=[x_1 \del_{x_2} - x_2 \del_{x_1}]$ for the angular derivative operator.    We first note the following. 

\begin{lemma}
\label{space-test-functions-dense}
For any $\lambda>0$, the space $C_c^\infty(\R^2)$ of test functions is dense in $(H,\langle \cdot,\cdot\rangle_{\lambda})$.  
\end{lemma}

\begin{proof}
The argument is essentially the same as the one proving the density of $C_c^\infty(\R^2)$ in $H^1(\R^2)$, see e.g. the proof of Theorem 9.2  in \cite{Brezisbook}. We only sketch it briefly. Let $W$ denote the subspace of functions in $H$ which vanish outside a bounded subset of $\R^2$. By a 
straightforward cut-off argument, $W$ is dense in $H$. Moreover, for a given function $u \in W$, it is well known that a sequence of mollifications $u_n \in C_c^\infty(\R^2)$ of $u$ converges to $u$ in the $H^1$-norm. Moreover, since there is a compact set $K \subset \R^2$ with the property that every $u_n$, $n \in \N$ vanishes in $\R^2 \setminus K$, the convergence in the $H^1$-norm also implies  convergence in $\|\cdot\|_\lambda$. This shows the claim.
\end{proof}

\medskip
Next, we consider the radial averaging operator 
\begin{equation}
  \label{eq:radial-averaging}
L^1_{loc}(\R^2) \to L^1_{loc}(\R^2), \quad u \mapsto u^\#, \qquad \text{with}\quad u^\# (x) := \frac{1}{2\pi}\int_{S^1} u(|x|\omega) \, d\omega \quad \text{for a.e. $x \in \R^2$.}  
\end{equation}
We note that, as a consequence of Jensen's inequality, the averaging operator extends to a continuous linear map $L^\rho(\R^2) \to L^{\rho}(\R^2)$ for every $\rho \in [1,\infty]$ with
\begin{equation}
  \label{eq:radial-averaging-L-rho-ineq}
|u^\#|_{\rho} \le |u|_{\rho} \qquad \text{for every $u \in L^{\rho}(\R^2)$.}
\end{equation}
Moreover, since $u^\# \in C_c^1(\R^2)$ for $u \in C_c^1(\R^2)$ and 
$$
\|u^\#\|_{\lambda}=\|u^\#\|_{H^1(\R^2)} \leq \|u\|_{H^{1}(\R^2)} \leq \|u\|_{\lambda} \quad \text{for $\lambda>0$,}
$$ 
the operator $u \mapsto u^\#$ extends to a continuous linear map $H \to H$.

We need the following angular Poincar\'e type estimates. 

\begin{lemma} 
\label{lemma1}  \,
\begin{enumerate} 
\item For any $u \in H$,
$$
|u|_{2}^2 \le |\del_\theta u|_2^2 + |u^\#|_{2}^2 .
$$
In particular, any $u \in H$ with $u^\# \equiv 0$ satisfies $|u|_{2}^2 \le |\del_\theta u|_2^2.$

\medskip
\item Let $\theta_0 \in (0,\pi)$ and consider the cone 
$$
C_{\theta_0} :=  \{ (r\cos \theta,r \sin \theta) \in \R^2\::\: r>0, |\theta| < \theta_0 \}.
$$

If $u \equiv 0$ on $\R^2 \setminus C_{\theta_0}$, then we have 
$$
|u|_{2} \le \frac{2\theta_0}{\pi}|\partial_\theta u|_{2}.
$$
\end{enumerate}
\end{lemma}

\begin{proof}
  (i) By Lemma~\ref{space-test-functions-dense}, it suffices to prove the claim for $u \in C^\infty_c(\R^2)$.

  We first assume that $u^\# \equiv 0$. In this case we have, in polar coordinates,
$$
|u|_{2}^2 = \int_{0}^\infty r \int_0^{2\pi} |u (r,\theta)|^2\,d\theta dr, 
$$
where the function $\theta \mapsto u(r,\theta)$ is $2\pi$-periodic and satisfies $\int_0^{2\pi} u(r,\theta)\,d\theta = 0$ for every $r>0$. Consequently, by {\it Wirtinger's inequality} for periodic functions,
$$
\int_0^{2\pi} |u (r,\theta)|^2\,d\theta \le \int_0^{2\pi} |\partial_\theta u (r,\theta)|^2\,d\theta \qquad \text{for every $r>0$,}
$$
which implies that 
$$
|u|_{2}^2 \le \int_{0}^\infty r \int_0^{2\pi} |\partial_\theta u (r,\theta)|^2\,d\theta dr = |\partial_\theta u|_2^2.
$$
If $u \in C^\infty_c(\R^2)$ is arbitrary, we may apply the above argument to the function $u-u^\#$. Since $(u- u^\#)^\#=0$ and $\langle u- u^\#, u^\# \rangle_{L^2(\R^2)} =0$, we get that 
$$
|u|_2^2 -|u^\#|_2^2 = |u-u^\#|_2^2 \le |\partial_\theta (u-u^\#)|_2^2 = |\partial_\theta u|_2^2,
$$
as claimed.

\medskip
(ii) Let $u \in H$ with $u \equiv 0$ on $\R^2 \setminus C_{\theta_0}$. By Lemma~\ref{space-test-functions-dense}, there exists a sequence $(u_n)_n$ in $C^\infty_c(\R^2)$ with $u_n \to u$.

\medskip
We fix $r_0>0$ and we let $\rho \in C^{\infty}([0,\infty))$ be a function with $0 \le \rho \le 1$, $\rho \equiv 0$ on $[0,r_0]$ and $\rho \equiv 1$ on $[2r_0,\infty)$. Moreover, we let $\theta' \in (\theta_0,\pi)$ and $\psi \in C^\infty_c(\R)$ be a function with $0 \le \psi \le 1$, $\psi \equiv 1$ in $[-\theta_0, \theta_0]$ and $\psi \equiv 0$ in $\R \setminus [-\theta',\theta']$. 
Next we define, in polar coordinates,
$$
\phi_0,\phi_1 \in L^\infty(\R^2) \cap C^\infty(\R^2), \qquad \phi_0(r,\theta)= \rho(r), \quad \phi_1(r,\theta) = \rho(r) \psi(\theta).
$$

Setting $v_n:= u_n \phi_1$ for $n \in \N$, it is then easy to see that 
\begin{equation}
  \label{eq:v-n-conv}
 v_n \to u \phi_1 = u \phi_0 \qquad \text{in $H$,}
\end{equation}
where the last equality follows since $u \equiv 0$ on $\R^2 \setminus C_{\theta_0}$. Moreover, we have, in polar coordinates,
$$
|v_n|_{2}^2 = \int_{0}^\infty r \int_{-\pi}^{\pi} |v_n (r,\theta)|^2\,d\theta dr,  
$$
where the function $\theta \mapsto v_n(r,\theta)$ is of class $C^1$ and satisfies $v_n(r,\theta)=0$ for $\theta \not \in [-\theta', \theta']$, $r>0$. Using again a classical Wirtinger type inequality (see section 1.7 in \cite{DYMMACKEAN1972}),  
$$
\int_{-\pi}^{\pi} |v_n (r,\theta)|^2\,d\theta \le \Bigl(\frac{2 \theta'}{\pi}\Bigr)^2 \int_{-\pi}^{\pi} |\partial_\theta v_n|^2 (r,\theta)\,d\theta \qquad \text{for every $r>0$,}
$$
which implies that 
\begin{equation}\label{eq:sequence-Wirtinger}
|v_n|_{2}^2 \le \Bigl(\frac{2 \theta'}{\pi}\Bigr)^2 \int_{0}^\infty r \int_{-\pi}^{\pi} |\partial_\theta v_n|^2 (r,\theta)\,d\theta dr = \Bigl(\frac{2 \theta'}{\pi}\Bigr)^2 |\partial_\theta v_n|_2^2
\end{equation}
for every $n \in \N$. 

\medskip
Using (\ref{eq:v-n-conv}), we may thus pass to the limit in \eqref{eq:sequence-Wirtinger} to obtain the inequality 
$$
|u \phi_0 |_{2}^2 \le \Bigl(\frac{2 \theta'}{\pi}\Bigr)^2 |\phi_0 \partial_\theta u |_2^2,
$$
which yields that 
$$
\|u\|_{L^2(\R^2 \setminus B_{2r_0}(0))} \le \frac{2\theta'}{\pi} 
\|\partial_\theta u\|_{L^2(\R^2)}.
$$
Since $r_0> 0$ and $\theta' >\theta_0$ were chosen arbitrarily, the claim follows.\end{proof}

Next we note embedding properties of the space $H$.

\begin{lemma}
\label{lemma2}
For every $\lambda>0$, $(H,\langle \cdot , \cdot \rangle_\lambda)$ is a Hilbert space canonically embedded in $H^1(\R^2)$. Moreover, $H$ is compactly embedded in $L^\rho(\R^2)$ 
for all $\rho \in (2,\infty)$.
\end{lemma}

\begin{proof}
We have 
$$
\|u\|_{H^1(\R^2)} \le  \|u\|_\lambda \qquad \text{for all $\lambda>0$, $v \in H$,}
$$
which implies that $H$ is a Hilbert space contained in $H^1(\R^2)$. By standard Sobolev embeddings, $H$ is thus embedded in $L^\rho(\R^2)$ 
for all $\rho \in [2,\infty)$. It remains to show that these embeddings are compact for $\rho>2$. 

Let $(u_n)_n$ be a sequence in $H$ with $u_n \weak 0$ in $H$, and suppose by contradiction that $u_n \not \to 0$ in $L^\rho(\R^2)$ for some $\rho>2$.

Since, $u_n \weak 0$ in $H^1(\R^2) $, it follows from Lions' Lemma \cite[Lemma I.1]{Lions} and Rellich's Theorem that, after passing to a subsequence, there exists a sequence 
$x^n \in \R^2$ with $|x^n| \to \infty$ and such that 
\begin{equation}
  \label{eq:w_n-w-weak-limit}
v_n \weak v \not = 0 \qquad \text{in $H^1(\R^2)$} 
\end{equation}
for the functions $v_n \in H^1(\R^2)$, $v_n= u_n( \cdot + x^n)$.  

\medskip
Let $r_n := |x^n|$. Passing to a subsequence, we may assume that the limits
$$
a := \lim \limits_{n \to \infty} \frac{x_1^n}{r_n}, \qquad  b: = \lim \limits_{n \to \infty} \frac{x_2^n}{r_n}
$$
exist, whereas $a^2+b^2=1$. For every $R>0$, we then have    
$$
 \begin{aligned}
\lambda^2 &\|u_n\|_{\lambda}^2 \ge   \int_{\R^2_+}|x_1 \partial_{x_2}u_n- x_2 \partial_{x_1}u_n|^2  dx =  \int_{\R^2}|(x_1+x_1^n) \partial_{x_2}v_n- (x_2+x_2^n) \partial_{x_1}v_n|^2  dx\\
\ge &  \int_{B_{R}(0)}|(x_1+x_1^n) \partial_{x_2}v_n- (x_2+x_2^n) \partial_{x_1}v_n|^2  dx\\
=&  r_n^2 \int_{B_{R}(0)}\left|\frac{x_1+x_1^n}{r_n} \partial_{x_2}v_n - \frac{x_2+x_2^n}{r_n} \partial_{x_1}v_n\right|^2  dx\\
\ge&  r_n^2 \Bigl( \int_{B_{R}(0)}|a \partial_{x_2}v_n - b\partial_{x_1}v_n|^2  dx\\
&- \sup_{x \in B_R(0)}\Bigl|\frac{x_1+x_1^n}{r_n}-a\Bigr| \|\partial_{x_2}v_n\|_{L^2(B_R(0))}^2 - \sup_{x \in B_R(0)}\Bigl|\frac{x_2+x_1^n}{r_n}-b\Bigr| \|\partial_{x_1}v_n\|_{L^2(B_R(0))}^2 \Bigr)\\
\ge&  r_n^2 \Bigl( \int_{B_{R}(0)}|a \partial_{x_2}v_n - b\partial_{x_1}v_n|^2  dx +o(1) \Bigr) \ge   r_n^2 \Bigl( \int_{B_{R}(0)}|a \partial_{x_2}v - b\partial_{x_1}v|^2  dx + o(1)\Bigr), 
\end{aligned}
$$
where in the last step we used the fact that
$$
a \partial_{x_2}v_n - b\partial_{x_1}v_n \;\weak \; a \partial_{x_2}v - b\partial_{x_1}v \qquad \text{in $L^2(B_R(0))$}
$$
and the weak lower semicontinuity of the $L^2$-norm. The boundedness of $(u_n)_n$ in $H$ now implies that
$$
\int_{B_{R}(0)}[a \partial_{x_2}v - b\partial_{x_1}v]^2  dx = 0 \qquad \text{for every $R>0$,}
$$
and thus 
\begin{equation}
  \label{eq:compactness-eq}
\int_{\R^2}|a \partial_{x_2}v - b\partial_{x_1}v|^2  dx  =0 .
\end{equation}
Since $a^2 + b^2 = 1$, if we had $a=0$ or $b=0$ it would follow that 
$$
\int_{\R^2} |\partial_{x_2}v|^2\,dx = 0\quad \text{or}\quad  \int_{\R^2} |\partial_{x_1}v|^2\,dx.
$$
The fact that $v\in L^2(\R^2)$ would imply $v \equiv 0$, contradicting \eqref{eq:w_n-w-weak-limit}. If, on the other hand, $a,b\neq 0$, \eqref{eq:compactness-eq} implies that 
$\del_{x_1} v = \frac{a}{b} \del_{x_2} v$ in $L^2(\R^2)$. Thus $v$ satisfies $\del_\beta v =0$ with $\beta = (1, - \frac{a}{b})$, which again implies $v \equiv 0$ and thus contradicts \eqref{eq:w_n-w-weak-limit}. The proof is finished.
\end{proof}

	\begin{lemma} \label{noncompactness-L-2}
		The embedding $H \embed L^2(\R^2)$ is not compact.
\end{lemma}

\begin{proof}
Let $\psi \in C^\infty_c((1,2)) \setminus \{ 0\}$.  After trivially extending $\psi$ to $\R$, for $n \in \N$ consider the functions
		$$
		u_n(r,s)=\frac{1}{\sqrt{r}} \psi(r-n)
		$$
		so that
		$$
		\supp \, u_n \subset \left\{ x \in \R_+^2: n+1 < |x| < n + 2 \right\} .
		$$
		Clearly, $u_n \weakto 0$ in $H$, but 
		\begin{align*}
		|u_n|_2^2 & =  2\pi \int_0^\infty  \psi(r-n)^2  \, dr
		= 2\pi \int_0^\infty  \psi(r)^2  \, dr  >0
		\end{align*}
		so $u_n \not \to 0$ in $L^2(\R^2)$. 
	\end{proof}

        In the following, we fix $p >2$ and $q=1$ in (\ref{eq:untransformed-problem whole space}), i.e., we consider the equation
\begin{equation}
\label{eq:untransformed-problem whole space:q=1}
\left \{
\begin{aligned}
-\Delta u - \frac{1}{\lambda^2} [x_1 \partial_{x_2} - x_2 \partial_{x_1}]^2 u + u &= |u|^{p-2} u &&\qquad \text{on $\R^2$,}\\
u(x) & \to 0 &&   \qquad \text{as $|x| \to \infty$.}
\end{aligned}
\right.  
\end{equation}
Here and in what follows, for a given $\lambda>0$, a function $u \in H$ will be called a weak solution of (\ref{eq:untransformed-problem whole space:q=1}) if 
$$
\langle u,v \rangle_\lambda = \int_{\R^2}|u|^{p-2}u v \,dx \qquad \text{for all $v \in H$.}
$$

As a consequence of Lemma~\ref{lemma2} and standard arguments in the calculus of variations, we see that for $\lambda>0$, the energy functional 
$$
E_\lambda: H \to \R, \quad
E_\lambda(u) := \frac{1}{2}\|u\|_\lambda^2-
\frac{1}{p}\int_{\R^2} |u|^p \,dx
$$
is of class $C^1$ and critical points of $E_\lambda$ are weak solutions of (\ref{eq:untransformed-problem whole space:q=1}). 

\medskip
We note the following uniform boundedness property of weak solutions of (\ref{eq:untransformed-problem whole space}). 

\begin{lemma} \label{L-infty regularity}
Fix $\lambda>0$ and let $u \in H$ be a weak solution of 
\begin{equation}
  \label{eq:L-infty-regularity-equation} 
	-\Delta u - \frac{1}{\lambda^2} \del_\theta^2 u + u = |u|^{p-2} u \quad \text{in $\R^2$}.
\end{equation}
	Then $u \in L^\infty(\R^2)$. Moreover, there exist constants $\sigma, C>0$, depending on $p>2$ but not on $u$ and $\lambda$, such that 
	\begin{equation} 
	|u|_\infty \leq C \|u\|_{H^1(\R^2)}^\sigma . \label{moser-ineq}
	\end{equation}
\end{lemma}

The fact that the constants $C$ and $\sigma$ in (\ref{moser-ineq}) do not depend on $\lambda$ is of key importance in the proofs of Theorem~\ref{main theorem}(i) and Theorem~\ref{non-radiality-L-infty-est-intro}(ii).  The proof of Lemma~\ref{L-infty regularity} follows by a Moser iteration scheme based on uniform estimates which do not depend on $\lambda>0$. We include the details in the appendix, see Lemma \ref{app: regularity} below.

\begin{remark}{\rm 
  \label{remark-l-infty-H-functions}
  If $f: \R \to \R$ is a $C^1$-function with $f(0)=0$ and $u \in H \cap L^\infty(\R^2)$, it is easy to see that also $f(u)= f \circ u \in H \cap L^\infty(\R^2)$ with
  $\nabla f(u) = f'(u) \nabla u$ and $\del_\theta f(u) = f'(u) \del_\theta u$.

By Lemma~\ref{L-infty regularity}, this observation applies, in particular, to weak solutions $u \in H$ of (\ref{eq:L-infty-regularity-equation}).}
\end{remark}

\medskip
Next we note that every nontrivial solution of \eqref{eq:untransformed-problem whole space:q=1} is contained in the Nehari manifold
$$
\cN_{\lambda} := \{ u \in H \setminus \{ 0 \} : E_\lambda'(u) u = 0 \} .
$$
Let 
\begin{equation} \label{mountain pass value}
\alpha_\lambda := \inf \limits_{u \in \cN_{\lambda}} E_\lambda(u)>0 ,
\end{equation}
then every minimizer is a critical point and hence a solution (cf. \cite{Szulkin-Weth} and Theorem~\ref{exist-mountain-pass-general} below). It is easy to see that such a minimizer is positive and thus radial by Theorem~\ref{symmetry of positive solutions}. Therefore, $\alpha=\alpha_{\lambda}$ does not depend on $\lambda$.

\medskip
Hence we now focus on sign-changing solutions. Consider
     \begin{align*} 
     \cM_\lambda  & := \left\{ u \in H: u^+ \not \equiv 0, u^- \not \equiv 0, \ E_\lambda'(u)u^+=E_\lambda'(u)u^-=0 \right\} \\
     & = \left\{ u \in H\setminus\{0\}: u^+, u^- \in \cN_{\lambda}\right\}
     \end{align*} 
     and set
     \begin{equation} \label{nodal minimal energy}
     \beta_\lambda := \inf_{ u \in \cM_\lambda} E_\lambda(u) .
     \end{equation}

\begin{proposition}\label{prop:beta-lambda}
The value $\beta_\lambda$ is positive. Moreover, every minimizer $u \in \cM_\lambda$ of (\ref{nodal minimal energy}) is a critical point of $E_\lambda$ and hence a sign-changing solution of \eqref{eq:untransformed-problem whole space:q=1}.
\end{proposition}

The proof  of Proposition \ref{prop:beta-lambda} follows the same argument as in the proof of Proposition 3.1 in \cite{Bartsch-Weth-Willem}.

\medskip
We also remark that $\beta_\lambda \geq 2\alpha>0$ in view of \eqref{mountain pass value} and the fact that for any $u\in H$,
$$
E_{\lambda}(u)=E_{\lambda}(u^+) + E_{\lambda}(u^-) \quad \hbox{and} \quad 
 E'_{\lambda}(u)u=E'_{\lambda}(u^+)u^+ + E'_{\lambda}(u^-) u^-.
$$

We say that a function $u \in H$ is a \emph{least energy nodal solution} of \eqref{eq:untransformed-problem whole space:q=1} if $u$ is a sign-changing solution of \eqref{eq:untransformed-problem whole space:q=1} such that $E_\lambda(u)=\beta_\lambda$. The following lemma yields the existence of a least energy nodal solution.
 
     \begin{lemma}
     	There exists $u \in \cM_\lambda$ such that $E_\lambda(u)=\beta_\lambda$.
     \end{lemma}
     \begin{proof}
     	We proceed similarly as in \cite{Castro-Cossio-Neuberger}. Let $(u_n)_n \subset \cM_\lambda$ be a minimizing sequence. Note that for any $u \in \cM_\lambda$ we have
     	$$
     	E_\lambda(u_n)= \left( \frac{1}{2} - \frac{1}{p} \right) \int_{\R^2} \left( |\nabla u|^2 + \frac{1}{\lambda^2} |\del_\theta u|^2 +  u^2 \right) \, dx, 
     	$$
     	which implies that $E_\lambda$ is coercive on $\cM_\lambda$. This yields that $(u_n)_n$ is bounded and we may therefore pass to a subsequence such that
$$
u_n \weakto u \quad \text{in $H$} .
$$
We then also have $u_n^\pm \weakto u^\pm$ in $H$, and the compact embedding $H \embed L^p$ implies
     	$$
     	\int_{\R^2} |u^\pm|^p \, dx= \lim \limits_{n\to \infty} \int_{\R^2} |u_n^\pm|^p \, dx = C \| u_n^\pm \|_{\lambda}^2 \geq C'>0. 
     	$$
     	Hence $u^\pm \not \equiv 0$.
     	
\medskip
Next, we show that $u_n^\pm \to u^\pm$ in $H$. Arguing by contradiction, assume first that $\|u^+\|_{\lambda}^2 < \liminf \limits_{n \to \infty} \|u_n^+\|_{\lambda}^2$. Then
$$
E_\lambda'(u^+)u^+ = \|u^+\|_{\lambda}^2 -  \|u^+\|_p^p < \liminf \limits_{n\to \infty}\left( \|u_n^+\|_{\lambda}^2 -  \|u_n^+\|_p^p \right) =0 .
$$

Hence the characterization of $\cN_{\lambda}$ yields the existence of $a\in (0,1)$ such that $a u^+ \in \cN_{\lambda}$. A similar argument yields $b u^- \in \cN_{\lambda}$ for some $0<b \leq 1$. Thus, $au^+ +bu^-\in \cM_\lambda$ and we estimate	
$$
\begin{aligned}
\beta_{\lambda} \leq & E_\lambda(a u^+ + b u^-)\\
< &\liminf \limits_{n\to \infty} E_\lambda(a u_n^+ + b u_n^-) = \liminf \limits_{n\to \infty}\left(  E_\lambda(a u_n^+) + E_\lambda(b u_n^-) \right)  
\\
\leq & \liminf \limits_{n\to \infty}(E_\lambda(u_n^+) + E_\lambda(u_n^-))
= \liminf \limits_{n\to \infty}E_\lambda(u_n) \\
=& \beta_{\lambda}, 
\end{aligned}
$$
which is a contradiction. Thus, after passing to a subsequence if necessary and using the uniform convexity of $(H,\|\cdot\|_{\lambda})$, we conclude that $u_n^+ \to u^+$ strongly in $H$. In particular, $u^+ \in \cN_\lambda$. Proceeding similarly, we prove that $u_n^- \to u^-$ strongly in $H$ and that $u^-\in \cN_{\lambda}$ and consequently, $u \in \cM_\lambda$  with $E_\lambda(u)=\beta_\lambda$.
\end{proof}
     
Summarizing the previous results, we have the following.
     \begin{corollary}\label{corollary1}
     	Let $p>2$. For every $\lambda>0$ there exists a least energy nodal solution to \eqref{eq:untransformed-problem whole space:q=1}, i.e. a sign-changing solution $u \in H$ such that $E_\lambda(u)=\beta_\lambda$.
     \end{corollary}


\begin{remark} \label{remark on generalization conditions}{\rm 
	We may also consider the more general equation
	\begin{equation} \label{eq:whole space generalized}
	\left \{
	\begin{aligned}
	-\Delta u - \frac{1}{\lambda^2} [x_1 \partial_{x_2} - x_2 \partial_{x_1}]^2 u + u &= f(u) &&\qquad \text{on $\R^2$,}\\
	u(x) & \to 0 &&   \qquad \text{for $|x| \to \infty$,}
	\end{aligned}
	\right.  
	\end{equation}
	where $f: \R \to \R$ is a continuous function. 
	In order to extend our results, consider the following conditions: 
	\begin{align*}
	(A_1) &  \ \text{There exists $C>0$ such that $|f(t)| \le C (|t| + |t|^p)$ for $t \in \R$} \\
	(A_2) & \ \text{$t \mapsto \frac{f(t)}{t}$ is strictly increasing on $\R \setminus \{0\}$ and $\lim_{t \to 0}\frac{f(t)}{t}\le 0, \ \lim_{t \to \pm \infty }\frac{f(t)}{t} = \infty.
		$} 
	\end{align*}
Under these assumptions, it can be shown that the results of this section, concerned with problem \eqref{eq:untransformed-problem whole space:q=1}, continue to hold true for \eqref{eq:whole space generalized}.}
\end{remark}

\section{Existence and symmetry of odd solutions}
\label{sec:exist-symm-odd}

This section is devoted to the study of solutions of the problem (\ref{eq:untransformed-problem}), which correspond, by odd reflection, to solutions of (\ref{eq:untransformed-problem whole space}) with hyperplane antisymmetry. In particular, we shall prove Parts (i) and (ii) of Theorem~\ref{exist-mountain-pass}. 

\medskip
Consider the space $H^+$ defined in (\ref{eq:def-H-plus}). For fixed $\lambda>0$ and $q \in \{0,1\}$, we endow $H^+$ with the $\lambda$-dependent scalar product 
	    $$
	    \langle u,v \rangle_{\lambda,q} \mapsto \int_{\R^2_+} \Bigl(\nabla u \cdot \nabla v + \frac{1}{\lambda^2}
	    (\del_\theta u)(\del_\theta v)+ q\, u v  \Bigr) dx,
	    $$
            and we let $\|\cdot\|_{\lambda,q}$ denote the corresponding norm.
Observe that any $u \in H^+$ can be extended to an element of $H$ either trivially or by odd reflection. Therefore, Lemma \ref{lemma1} and \ref{lemma2} immediately yield the following. 
\begin{corollary} \label{H+ properties}
	\begin{itemize}
		\item[(i)]
		Any $u \in H^+$ satisfies
     	\begin{equation} \label{angular - L2 estimate}
    	|u|_{2}^2 \le \int_{\R^2_+}|\del_\theta u|^2  dx .
	    \end{equation}
            In particular, the norms $\|\cdot\|_{\lambda,0}$ and $\|\cdot\|_{\lambda,1}$ are equivalent on $H^+$, and $H^+$ is a Hilbert space with either of these norms. Moreover, we have a continuous embedding $H^+ \hookrightarrow H^1(\R^2_+)$.
	    \item[(ii)]
The space $H^+$ is compactly embedded in $L^\rho(\R^2_+)$ for $\rho>2$.
	\end{itemize}
\end{corollary}

\begin{remark}{\rm 
  (i) Similar statements are also true, when the underlying space is the cone $C_{\theta_0}$ described in Lemma \ref{lemma1}.
  \medskip
  
  (ii) As in Lemma~\ref{noncompactness-L-2}, we see that the embedding $H^+ \hookrightarrow L^2(\R^2_+)$ is not compact.}
\end{remark}

First, we establish the symmetry of postive weak solutions of \eqref{eq:untransformed-problem} as a consequence of the following.

\begin{theorem} \label{General half space symmetry}
	Let $\lambda>0$, and let $f \in C^1([0,\infty))$ satisfy 
\begin{equation}\label{eq:growth-der-f}
f'(t) \le C \Bigl(t^{\sigma_1}+t^{\sigma_2}\Bigr) \quad \hbox{for} \quad t \ge 0
\end{equation}
with constants $\sigma_1,\sigma_2>0$. Moreover, let $u \in H^+ \cap L^\infty(\R^2)$ be a positive weak solution of the problem 
		\begin{equation} \label{eq:General half space equation-symmetry}
		\left \{	\begin{aligned}
		-\Delta u - \frac{1}{\lambda^2} \del_\theta^2 u &=f(u) &&\qquad \text{on $\R^2_+$,}\\
		u &=0 && \qquad \text{on $\partial \R^2_+$.}
		\end{aligned}
		\right.  
		\end{equation}
Then $u$ is symmetric with respect to the $x_1$-axis and decreasing with respect to the angle $|\theta|$ from the $x_1$-axis.
\end{theorem}

\begin{remark}
  \label{remark-after-symmetry}{\rm
Theorem~\ref{General half space symmetry} in particular applies in the case where the nonlinearity $f$ is given by $f(t)=-q t + |t|^{p-2}t$ for some $p \in (2,\infty)$, $q \in \{0,1\}$. In this case, Lemma~\ref{L-infty regularity} and Remark~\ref{Appendix-remark} below imply that every weak solution $u \in H^+$ of (\ref{eq:growth-der-f}) is bounded. Hence we deduce the statement of Theorem~\ref{exist-mountain-pass}(ii). }
\end{remark}

\begin{proof}[Proof of Theorem~\ref{General half space symmetry}]
	For simplicity, we assume $\lambda =1$. We shall argue by the method of rotating planes. 
	For $\theta \in [-\frac{\pi}{2},0) \cup (0,\frac{\pi}{2}]$, set $e_\theta := (\cos \theta, \sin \theta )$,
	$$
	 T_\theta := \{x \in \R^2: x \cdot e_\theta =0\} \quad \hbox{and}\quad \Sigma_\theta := \{ x \in \R_+^2: x \cdot e_\theta <0 \} .
$$
	
Given a positive solution $u\in H^+\cap L^{\infty}(\R^2_+)$ of \eqref{eq:General half space equation-symmetry}, consider the functions $u_\theta, w_\theta:\Sigma_\theta  \to \R$ defined by
$$
u_\theta(x) = u(x-2 (x \cdot e_\theta)e_\theta) \quad \hbox{and} \quad w_\theta := u_\theta - u
$$
and extend them trivially outside $\Sigma_{\theta}$.	
	
\medskip
A direct calculation shows  that $w_\theta$ satisfies 
\begin{equation}
  \label{eq:w-theta-sat}
	\begin{aligned}
	-\Delta w_\theta - \del_\theta^2 w_\theta & = c_\theta(x) w_\theta \quad && \text{in $\Sigma_\theta$} \\
	w_\theta & =  0 \quad \quad  && \text{on $T_\theta$} \\
	w_\theta &>0 \quad \quad  && \text{on $\del \Sigma_\theta \setminus T_\theta$,} 
	\end{aligned}
\end{equation}
	where 
	$$
	c_\theta(x)=\int_0^1 f'\bigl((1-t)u(x)+t u_\theta(x)\bigr)\,dt. 
	$$
		
Consider the set
$$
\Theta^+ := \left\{ \theta \in \left(0,\frac{\pi}{2} \right): w_{\theta} \geq 0 \text{ in $\Sigma_{\theta}$} \right\}
$$
which is clearly a 
closed set in $(0,\frac{\pi}{2})$.

\medskip
We claim that $\Theta^+$ is non-empty. To prove this claim, we proceed as follows. Observe first that $w_\theta^- := \min \{w_\theta,0\} \in H^+$. Moreover, using \eqref{eq:growth-der-f},  we have that for $x \in \Sigma_\theta$ with $w_\theta^-(x)<0$,
\begin{equation}\label{eq:estimate-ctheta}
\begin{aligned}
c_\theta(x) \le & C \int_0^1 \Bigl[\bigl((1-t)u(x)+t u_\theta(x)\bigr)^{\sigma_1}+
\bigl((1-t)u(x)+t u_\theta(x)\bigr)^{\sigma_2} \Bigr]dt \\
\le & C \Bigl[u^{\sigma_1}(x)+u^{\sigma_2}(x)\Bigr].
\end{aligned}
\end{equation}

Also, the boundary conditions imply $w_\theta^- \equiv 0$ on $\del \Sigma_\theta$, and testing the equation \eqref{eq:w-theta-sat} against $w_\theta^-$ yields
\begin{equation}
  \label{eq:testing-ineq}
  \begin{aligned}
	|\nabla w_\theta^-|_2^2 + |\del_\theta w_\theta^-|_2^2 = & \int_{\R^2} c_\theta(x) (w_\theta^-)^2 \, dx\\
	 \leq & C \int_{\R^2}\Bigl[u^{\sigma_1}+u^{\sigma_2}\Bigr]
(w_\theta^-)^2\,dx\\ 
\le & C_0 |w_\theta^-|_2^2
\end{aligned}
\end{equation}
with $C_0 = C\bigl(|u|_\infty^{\sigma_1} +|u|_\infty^{\sigma_2}\bigr)$. Therefore, by Lemma~\ref{lemma1}(ii), 
$$
\frac{\pi}{2\theta}|w_\theta^-|_2 \le |\del_\theta w_\theta^-|_2 \le \sqrt{C_0}|w_\theta^-|_2.  
$$
Consequently, $w_\theta^- \equiv 0$ provided that $0 < |\theta| < \frac{\sqrt{C_0} \pi}{2}$ and this proves the claim. 

\medskip
Next, we claim that $\Theta^+$ is also open in $(0,\frac{\pi}{2})$. To see this, let $\theta_0 \in \Theta^+$. Since $w_{\theta_0} \not \equiv 0$ by (\ref{eq:w-theta-sat}), the strong maximum principle implies that $w_{\theta_0}>0$ in $\Sigma_{\theta_0}$. 

\medskip
Fix $\rho>2$ such that $\tau_i:= \frac{\sigma_i \rho}{\rho-2}>2$ for $i=1,2$. By Lemma~\ref{lemma2}, there exists $\kappa_\rho>0$ such that 
$$
|w|_\rho^2 \le \kappa_\rho \Bigl(|\nabla w|_2^2 + |\del_\theta w|_2^2\Bigr) \qquad \text{for all $w \in H^+$.}
$$
Moreover, we may choose a compact set $D \subset \Sigma_{\theta_0}$ such that
	$$
	\|u\|_{L^{\tau_1}(\Sigma_{\theta_0} \setminus D)}^{\sigma_1}+\|u\|_{L^{\tau_2}(\Sigma_{\theta_0} \setminus D)}^{\sigma_2}< \frac{1}{2\kappa_\rho C},
	$$
where $C>0$ is the constant in \eqref{eq:estimate-ctheta}.

\medskip	
On the other hand, by continuity of the family $w_{\theta}$ w.r.t. $\theta$ there exists a neighborhood $N \subset (0,\frac{\pi}{2})$ of $\theta_0$ with the property that for all $\theta \in N$,
$$
w_\theta>0\quad \text{in $D$}\quad \text{and}\quad 
\|u\|_{L^{\tau_1}(\Sigma_{\theta} \setminus D)}^{\sigma_1}+\|u\|_{L^{\tau_2}(\Sigma_{\theta} \setminus D)}^{\sigma_2} < \frac{1}{2\kappa_\rho C}.
$$
From (\ref{eq:testing-ineq}) and H\"older's inequality, it  follows that  
\begin{align*}
	|w_\theta^-|_{\rho}^2 &\leq \kappa_\rho \Bigl(|\nabla w_\theta^-|_2^2 + |\del_\theta w_\theta^-|_2^2\Bigr)\le \kappa_\rho C \int_{\R^2}\Bigl[u^{\sigma_1}+u^{\sigma_2}\Bigr](w_\theta^-)^2\,dx\\
&\le \kappa_\rho C \Bigl(\|u\|_{L^{\tau_1}(\Sigma_{\theta_0} \setminus D)}^{\sigma_1}+\|u\|_{L^{\tau_2}(\Sigma_{\theta_0} \setminus D)}^{\sigma_2}\Bigr)|w_\theta^-|_{\rho}^2 \le \frac{1}{2}|w_\theta^-|_{\rho}^2
\end{align*}
for any $\theta \in N$.

\medskip
Consequently, $w_\theta^- \equiv 0$ for $\theta \in N$ and this proves the claim. 

\medskip
Since $\Theta^+$ is an open, closed and nonempty subset of $\left(0,\frac{\pi}{2}\right)$, we conclude that $\Theta^+ = \left(0,\frac{\pi}{2}\right)$. In the same manner, we see that  
$$
	\Theta^- := \left\{ \theta \in \left(-\frac{\pi}{2},0 \right): w_{\theta} \geq 0 \text{ in $\Sigma_{\theta}$} \right\} = (-\frac{\pi}{2},0)
$$ 
Consequently $u$ is decreasing with respect to the angle $|\theta|$ from the $x_1$-axis.

Finally, a continuity argument also shows that $w_\theta \ge 0$ in $\Sigma_\theta$ for $\theta \in \{\pm \frac{\pi}{2}\}$, which, in particular, forces the symmetry of $u$ with respect to reflection at the $x_1$-axis.
\end{proof}

Next, let $f: \R \to \R$ be a continuous function satisfying $(A_1)$ and $(A_2)$ as in Remark \ref{remark on generalization conditions} and set  $F(t)=\int_0^t f(s) \, ds$.
We consider the energy functional 
$$
E_{\lambda}^+: H^+ \to \R, \quad
E_{\lambda}^+(u) := \frac{1}{2} \|u\|_{\lambda,0}^2- \int_{\R^2_+} F(u) \,dx 
$$

Again, standard arguments in the calculus of variations show that $E_{\lambda}^+$ is of class $C^1$, and critical points of $E_{\lambda}^+$ are solutions of
the associated Euler-Lagrange equation
		\begin{equation} \label{eq:General half space equation}
		\left \{	\begin{aligned}
		-\Delta u - \frac{1}{\lambda^2} \del_\theta^2 u &=f(u) &&\qquad \text{on $\R^2_+$,}\\
		u &=0 && \qquad \text{on $\partial \R^2_+$.}
		\end{aligned}
		\right.  
		\end{equation}
As in Section \ref{sec:funct-analyt-fram} we consider the associated Nehari manifold
$$
\cN_{\lambda}^+ := \left\{ u \in H^+ \setminus \{0\}: [E_{\lambda}^+]'(u)u=0 \right\} 
$$
and set 
\begin{equation}
  \label{eq:inf-sup-mountain-pass}
c_{\lambda}:=\inf_{u \in \cN^+} E_{\lambda}^+(u) .
\end{equation}
This is the ground state energy in the sense that $E_{\lambda}^+(u) \ge c_{\lambda}$ for every nontrivial solution of (\ref{eq:General half space equation}). 
\begin{theorem}
\label{exist-mountain-pass-general}  
Let $p>2$, $\lambda>0$, and assume that $f: \R \to \R$ is a continuous function satisfying the assumptions $(A_1)$ and $(A_2)$ listed in Remark \ref{remark on generalization conditions}.
Then 
\begin{equation}
  \label{eq:c-lambda-q-characterization}
c_{\lambda} = \inf_{u \in H^+ \setminus \{0\}} \sup_{t \ge 0} E_{\lambda}^+(t u).
\end{equation}
Moreover, problem \eqref{eq:General half space equation} admits a ground state solution, i.e., a solution $v \in H^+ \setminus \{0\}$ such that $E_{\lambda}^+(v)=c_{\lambda}$.
\end{theorem}

\begin{proof}
The proof essentially follows the lines of the proof of \cite[Theorem 20]{Szulkin-Weth}, see also \cite[Section 4]{Liu-Wang}. We note here that $(A_1)$ and $(A_2)$ ensure that the assumptions in \cite[Theorem 20]{Szulkin-Weth} are satisfied. Indeed, $(A_2)$ implies that for any $R>0$ there exists $t_R>0$ such that $f(t) \geq Rt$ for $t \geq t_R$. Thus
$$
F(t)=\int_0^t f(s) \, ds \geq \int_{t_R}^{t} Rs \, ds = \frac{R}{2} (t^2 - t_R^2) 
$$
for $t \geq t_R$. It follows that
$$
\lim_{t \to \infty} \frac{F(t)}{t^2} = \infty ,
$$
i.e. assumption (iv) in \cite[Theorem 20]{Szulkin-Weth} is satisfied. Consequently, the proof given there can be carried through similarly, with some simplifications because the compact embedding $H^+ \hookrightarrow L^p(\R^2_+)$ replaces arguments based on compactness modulo translations in the periodic setting of \cite[Theorem 20]{Szulkin-Weth}.
\end{proof}

\begin{remark}{\rm 
    (i) The statement of Theorem~\ref{exist-mountain-pass}(i) is a special case of Theorem~\ref{exist-mountain-pass-general}, since the nonlinearity $t \mapsto f(t)=-q t + |t|^{p-2}t$ satisfies conditions $(A_1)$ and $(A_2)$ if $q \in \{0,1\}$ and $p \in (2,\infty)$.

    (ii) Under the assumptions of Theorem~\ref{exist-mountain-pass-general}, it can be shown that ground state solutions cannot change sign, see \cite[Remark 17]{Szulkin-Weth}.}
\end{remark}

\section{Asymptotics of least energy odd solutions}
\label{sec:asympt-lambda-to}
  \label{sec:asympt-via-energy-lemma}

In this section we fix $p \in (2,\infty)$, $q=1$, and we study the asymptotics of least energy solutions to (\ref{eq:untransformed-problem}) in the case $q=1$ as $\lambda \to \infty$ and as $\lambda \to 0$. In particular, we shall complete the proofs of Theorem~\ref{exist-mountain-pass}(iii) and of Theorem~\ref{rescaled convergence}. We will use the notation introduced in the previous section in the special case of the nonlinearity $t \mapsto f(t)=-t + |t|^{p-2}t$ which satisfies conditions $(A_1)$ and $(A_2)$. By the definition of the mountain pass value in (\ref{eq:inf-sup-mountain-pass}) and the fact that $E_{\lambda_1}^+ \ge E_{\lambda_2}^+$ for $0 <\lambda_1< \lambda_2 < \infty$, we infer that the function 
$$
(0,\infty) \to (0,\infty), \qquad \lambda \mapsto c_{\lambda}
$$
is decreasing, and therefore the limits
\begin{equation}
  \label{eq:def-c-limits}
c_0 := \lim_{\lambda \to 0} c_{\lambda} \qquad \text{and} \qquad c_\infty := \lim_{\lambda \to \infty} c_{\lambda}
\end{equation}
exist in $[0,\infty]$. Next we note that 
\begin{equation}
  \label{eq:ray-characterization}
\sup_{t \ge 0}E_{\lambda}^+ (tv) = E_{\lambda}^+(t_v^\lambda v)= \Bigl(\frac{1}{2}-\frac{1}{p}\Bigr)\frac{
\|v\|_{\lambda,1}^{\frac{2p}{p-2}}}{|v|_{p}^{\frac{2p}{p-2}}} \quad \text{for every $v \in H^+ \setminus \{0\}$}
\end{equation}
with 
$$
t_v^\lambda = \left(\frac{\|v\|_{\lambda,1}^2}{|v|_{p}^p}\right)^{\frac{1}{p-2}}.
$$
We start by considering the asymptotics of least energy solutions to (\ref{eq:untransformed-problem}) as $\lambda \to \infty$. 

\subsection{The limit $\lambda \to \infty$}
\label{sec:limit-lambda-to-infty}
Consider the limit energy functional 
$$
E_*: H^1(\R^2) \to \R, \qquad E_*(v)= \frac{1}{2} \int_{\R^2} \left( |\nabla v|^2 + v^2 \right) \,dx - \frac{1}{p} \int_{\R^2} |v|^p \,dx.
$$
Similarly as in (\ref{eq:ray-characterization}), for $v \in H^1(\R^2) \setminus \{0\}$ we have 
\begin{equation}\label{eq:energy-infty-lambda}
\sup_{t \ge 0}E_*(t v) = E_*(t_v v)= \Bigl(\frac{1}{2}-\frac{1}{p}\Bigr)\frac{\|v\|_{H^1(\R^2)}^{\frac{2p}{p-2}}}{|v|_{p}^{\frac{2p}{p-2}}}
\end{equation}
with $t_v =\left(\frac{\|v\|_{H^1(\R^2)}^2}{|v|_{p}^{p}}\right)^{\frac{1}{p-2}}$.

\medskip
Observe that for every $v\in H^1(\R^2)$ with $E_*'(v)v=0$ we have $t_{v}=1$ and hence
$$
\sup_{t \ge 0} E_*(t v)= E_*(v).
$$
Define
\begin{equation}\label{eq:def-limit ground state}
\hat{c}_\infty := \inf_{v \in H^1(\R^2) \setminus \{0\}} \sup_{t \ge 0} E_*(t v)
\end{equation}
and let $w_\infty$ denote the unique positive radial solution (see \cite{Kwong}) of the problem 
\begin{equation}\label{eq:limit-infty}
-\Delta w_\infty + w_\infty = |w_\infty|^{p-2}w_\infty, \qquad  w_{\infty} \in C^2(\R^2) \cap H^1(\R^2). 
\end{equation}

Since $E_*'(w_{\infty})w_{\infty}=0$, $t_{w_{\infty}}=1$ and hence 
\begin{equation}
  \label{eq:label-w-infty-saddle}
\sup \limits_{t\geq 0}E_*(tw_{\infty})=E_*(w_{\infty}).
\end{equation}
The following result provides a variational characterization of the limit $c_\infty$, defined in \eqref{eq:def-c-limits}, in terms of $\hat{c}_{\infty}$ and $w_{\infty}$.

\begin{lemma}
\label{energy-asymptotics}
\begin{equation}
 \label{eq:limit-c-infty}
c_\infty = \hat{c}_{\infty}= E_*(w_\infty). 
\end{equation}
\end{lemma}

\begin{proof}
We first prove the second equality in \eqref{eq:limit-c-infty}. Since the proof is standard, we only sketch the argument. By (\ref{eq:label-w-infty-saddle}), we have $\hat{c}_{\infty}\leq E_*(w_{\infty})$. On the other hand, using Schwarz symmetrization and \eqref{eq:energy-infty-lambda}, it is easy to see that
$$
\hat{c}_\infty = \inf_{v \in H^1_{rad}(\R^2) \setminus \{0\}} \sup_{t \ge 0} E_*(t v).
$$
Proceeding as in Theorem 20 and Remark 17 in \cite{Szulkin-Weth} and using the compactness of the embedding $H^1_{rad}(\R^2)\hookrightarrow L^{p}(\R^2)$, one can prove that $\hat{c}_{\infty}$ is attained at a positive radial solution of \eqref{eq:limit-infty}.  By uniqueness, we then deduce that $\hat{c}_{\infty}=E_*(w_{\infty})$.
\medskip

Next, we prove the first equality in \eqref{eq:limit-c-infty}. Identifying $v \in H^+$ with its trivial extension in $H$, we see that $E_{\lambda}^+(v)= E_{\lambda}(v) \ge E_*(v)$ for any $v \in H^+$ and any $\lambda>0$. Hence $c_{\lambda}\geq \hat{c}_{\infty}$ for any $\lambda>0$ by (\ref{eq:c-lambda-q-characterization}) and (\ref{eq:def-limit ground state}). Taking the limit as $\lambda \to \infty$, we obtain that $c_{\infty}\geq \hat{c}_{\infty}$.

\medskip
To see the opposite inequality, we let $v \in H^1(\R^2) \setminus \{0\}$ be arbitrary. Let $t_v>0$ be as in \eqref{eq:energy-infty-lambda}, which implies that 
$$
0 = \frac{\partial_t \big|_{t_v} E_*(t v)}{t_v} = \|v\|_{H^1(\R^2)}^2- t_v^{p-2} \int_{\R^2} |v|^{p} \,dx. 
$$
From this we find that
$$
\|v\|^2_{H^1(\R^2)} < (2t_v)^{p-2} \int_{\R^2} |v|^p \,dx.
$$
Since $C_c^\infty(\R^2)$ is dense in $H^1(\R^2)$, there exists a sequence $\psi_n \in C_c^\infty(\R^2)$ such that $\|v-\psi_n\|_{H^1(\R^2)} \to 0$ as $n \to \infty$, and 
$$
\|\psi_n\|^2_{H^1(\R^2)} < (2t_v)^{p-2} \int_{\R^2} | \psi_n|^p \,dx \qquad \text{for all $n \in \N$.} 
$$
This implies that 
\begin{equation}
  \label{eq:energy-limit-proof-1}
\sup_{t \ge 0} E_*(t \psi_n)= \sup_{0 \le t \le 2t_v} E_*(t \psi_n) \to \sup_{0 \le t \le 2t_v} E_*(t v)= E_*(t_v v) \qquad \text{as $n \to \infty$.} 
\end{equation}
Next, we fix $n \in \N$ and choose $y_n \in \R^2$ such that $\tilde \psi_n \in C^\infty_c(\R^2_+) \subset H^+$ for the function $\tilde \psi_n: \R^2_+ \to \R$, $\tilde \psi_n(x)= \psi_n(x-y_n)$. Then there exists ${t}_n>2t_v$ such that 
$$
\|\psi_n\|^2_{\lambda,1} = \|\psi_n\|_{H^1(\R^2_+)}^2 +\frac{1}{\lambda^2} \|\del_\theta \psi_n\|_{L^2(\R^2_+)}^2 < (2{t}_n)^{p-2} \int_{\R^2} |\psi_n|^p \,dx \qquad \text{for all $\lambda \ge 1$.} 
$$
Using the fact that 
$$
\frac{t^2}{\lambda^2} \int_{\R^2_+}|\del_\theta \psi_n|^2 \, dx \to 0 \quad \text{as $\lambda \to \infty$ uniformly in $t \in [0,t_n]$,}
$$
 we find that 
\begin{align}
c_\infty &= \lim_{\lambda \to \infty} c_{\lambda} \le \lim_{\lambda \to \infty} \sup_{t \ge 0} E_{\lambda}^+(t \tilde \psi_n) =
\lim_{\lambda \to \infty} \sup_{0 \le t \le t_n} E_{\lambda}^+(t \tilde \psi_n) \nonumber\\
&= \sup_{0 \le t \le t_n} E_*(t \tilde \psi_n)=\sup_{t \ge 0} E_*(t \tilde \psi_n)=\sup_{t \ge 0} E_*(t \psi_n),   \label{eq:energy-limit-proof-2}
\end{align}

Combining (\ref{eq:energy-limit-proof-1}) and (\ref{eq:energy-limit-proof-2}), it follows that 
$$ 
c_\infty \le E_*(t_v v) = \sup_{t \ge 0} E_*(t v).
$$
Since $v \in H^1(\R^2) \setminus \{0\}$ was arbitrary, we conclude that $c_{\infty}\leq \hat{c}_{\infty}$. This completes the proof of the theorem.
\end{proof}
Now we are in a position to prove Theorem \ref{exist-mountain-pass}.
\begin{proof}[Proof of Theorem~\ref{exist-mountain-pass}.] 
The existence statement in (i) is a direct consequence of Theorem \ref{exist-mountain-pass-general}, whereas the symmetry property stated in Theorem~\ref{exist-mountain-pass}~(ii) is a special case of Theorem \ref{General half space symmetry}.

\medskip
Next, we prove the asymptotics in (iii). In what follows, the functions in $H^+$ are extended trivially outside $\R^2_+$. Assume that $1\leq \lambda_{k} \to \infty$ and, for every $k\in \N$, let $u_{k}\in H^+$ denote a positive least energy solution of \eqref{eq:untransformed-problem} for $\lambda=\lambda_k$. Observe that for $k \in \N$, $$
\|u_k\|_{\lambda_k,1}^2= |u_k|_{p}^p
$$
and   
\begin{align*}
c_{1} \ge c_{\lambda_k} = E_{\lambda_k}^+(u_k) &= \Bigl(\frac{1}{2}-\frac{1}{p}\Bigr)\|u_k\|_{\lambda_k,1}^2 
= \Bigl(\frac{1}{2}-\frac{1}{p}\Bigr) |u_k|_{p}^p \ge c_\infty>0.
\end{align*}

Since 
$$
\|u_k\|_{H^1_0(\R^2_+)}^2 \le \|u_k\|_{\lambda_k,1}^2 \qquad \text{for every $k \in \N$,}
$$
we conclude that $(u_k)_k$ is bounded in $H^1_0(\R^2_+) \subset H^1(\R^2)$. Moreover, $|u_k|_{p}$ remains bounded away from zero. From Lions' Lemma \cite[Lemma I.1]{Lions} and Theorem~\ref{General half space symmetry}, it thus follows that, after passing to a subsequence, there exists a sequence of numbers 
$\tau_k \in (0,\infty)$ such that $w_k \weak w \not = 0$ in $H^1(\R^2)$ for the functions $w_k:= u_k( \cdot + (\tau_k,0))$. Observe that $w\geq 0$ a.e. in $\R^2$.

\medskip
We first claim that 
\begin{equation}
  \label{eq:tau-k-infty}
\tau_k \to \infty \qquad \text{as $k \to \infty$.}  
\end{equation}
Indeed, suppose by contradiction that $(\tau_k)_k$ contains a bounded subsequence.  Then we may again pass to a subsequence with the property that 
$$
u_k \weak u \not =0 \qquad \text{in $H^1_0(\R^2_+)$,}
$$ 
where $u\geq 0$ a.e. in $\R^2_+$. For $\phi \in C^\infty_c(\R^2_+)$ and $R>0$ with $\supp \, \phi \subset B_R(0)$ we then have  
$$
\frac{1}{\lambda_k^2} \int_{\R^2_+} (\del_\theta u_k)(\del_\theta \phi) dx
\le  \frac{R^2}{\lambda_k^2} \|\nabla u_k\|_{L^2(R^2_+)} \| \nabla \phi\|_{L^2(R^2_+)} \to 0 \qquad \text{as $k \to \infty$} 
$$
and thus 
$$
\int_{\R^2_+} \Bigl( \nabla u \cdot \nabla \phi + u \phi -u^{p-1} \phi\Bigr)dx = \lim_{k \to \infty}\Bigl(\langle u_k,\phi \rangle_{\lambda_k,1}  - \int_{\R^2_+} u_k^{p-1} \phi\, dx\Bigr)= 0.   
$$
Hence $u \in H^1_0(\R^2_+)$ is a nontrivial nonnegative weak solution of the problem 
$$
-\Delta u + u = u^{p-1}\quad \text{in $\R^2_+$},\qquad u = 0 \quad \text{on $\del \R^2_+$}
$$
which contradicts a classical nonexistence result of Esteban and Lions in \cite{Esteban-Lions}. Thus (\ref{eq:tau-k-infty}) is true.

\medskip
We now claim that 
\begin{equation}
  \label{eq:tau-k-infty-refined}
\frac{\tau_k}{\lambda_k} \to 0 \qquad \text{as $k \to \infty$.}  
\end{equation}
Before proving the claim, observe that by weak lower semicontinuity,
\begin{align}
  \tau_k^{-2}& \int_{\R^2_+}|\del_\theta u_k|^2dx =  \tau_k^{-2} \int_{\R^2_+}|x_1 \partial_{x_2}u_k- x_2 \partial_{x_1}u_k|^2  dx \nonumber\\
= & \tau_k^{-2} \int_{\R^2}|(x_1+\tau_k) \partial_{x_2}w_k- x_2 \partial_{x_1}w_k|^2  dx \nonumber \ge \int_{B_{R}(0)}\bigl|\frac{x_1+\tau_k}{\tau_k} \partial_{x_2}w_k - \frac{x_2}{\tau_k} \partial_{x_1}w_k\bigr|^2  dx\nonumber\\
\ge &  \int_{B_{R}(0)} |\partial_{x_2}w|^2 dx + o(1) \qquad \text{for every $R>0$,} \label{neweq-1}
 \end{align}
whereas for $R>0$ large enough, 
$$
\int_{B_{R}(0)} |\partial_{x_2}w|^2 dx > 0 
$$
since $w \in H^1_0(\R^2_+)$ is not identically zero.

\medskip
Now, in order to prove \eqref{eq:tau-k-infty-refined}, assume by contradiction that, passing to a subsequence,  
$$
\frac{\tau_k}{\lambda_k} \to d \in (0,\infty] \qquad \text{as $k \to \infty$.}
$$
In the case where $d=\infty$ the estimate (\ref{neweq-1}) implies that 
$$
\frac{1}{\lambda_k^2} \int_{\R^2_+}|\del_\theta u_k|^2  dx  \to \infty  \qquad \text{as $k \to \infty$}
$$
and therefore 
$$
\|u_k\|_{\lambda_k,1} \to \infty \qquad \text{as $k \to \infty$}
$$
which contradicts the fact that $\|u_k\|_{\lambda_k,1}$ is bounded in $k$.

\medskip

Therefore we have $d< \infty$ and from \eqref{neweq-1}, 
\begin{equation}
\label{neweq-2}
\liminf \limits_{k \to \infty} \frac{1}{\lambda_k^2} 
\int_{\R^2_+}|\del_\theta u_k|^2dx \ge d^2  \int_{\R^2} |\partial_{x_2}w|^2 dx.
\end{equation}
Notice that in this case, $w \in H^1(\R^2)$ is a weak solution of 
\begin{equation}
  \label{eq:equation-w-limit-c-pos}
-\Delta w + d^2  \partial_{x_2 x_2} w + w = w^{p-1} \qquad \text{on $\R^2$.}
\end{equation}
Indeed, let $\phi \in C^\infty_c(\R^2)$ and let $\phi_k \in C^\infty_c(\R^2_+)$ be defined by 
$$
\phi_k(x_1,x_2) = \phi(x_1- \tau_k,x_2)
$$
for $k$ sufficiently large. We then have 
\begin{align*}
\frac{1}{\lambda_k^2} &\int_{\R^2_+} (\del_\theta u_k)(\del_\theta \phi_k)  dx
=\frac{(d^2+o(1))}{\tau_k^2} \int_{\R^2_+} (x_1 \partial_{x_2} u_k- x_2 \partial_{x_1} u_k)(x_1 \partial_{x_2} \phi_k-x_2 \partial_{x_1} \phi_k)  dx\\
&=(d^2+o(1)) \int_{\R^2} \Bigl(\frac{x_1+\tau_k}{\tau_k} \partial_{x_2} w_k- \frac{x_2}{\tau_k} \partial_{x_1} w_k\Bigr)\Bigl(\frac{x_1+\tau_k}{\tau_k} \partial_{x_2} \phi-\frac{x_2}{\tau_k} \partial_{x_1} \phi\Bigr)  dx\\
&=d^2 \int_{\R^2} \partial_{x_2} w \partial_{x_2} \phi dx + o(1) \qquad \text{as $k \to \infty$}
\end{align*}
and therefore 
\begin{align*}
&\int_{\R^2_+} \Bigl( \nabla w \cdot \nabla \phi + d^2  \partial_{x_2} w \partial_{x_2} \phi + w \phi -w^{p-1} \phi\Bigr)dx\\
  &= \lim_{k \to \infty} \int_{\R^2_+} \Bigl( \nabla u_k \cdot \nabla \phi_k + \frac{1}{\lambda_k^2}(\del_\theta u_k)(\del_\theta \phi_k)+ u_k \phi_k -u_k^{p-1} \phi_k \Bigr)dx\\
&= \lim_{k \to \infty}\Bigl( \langle u_k,\phi  \rangle_{\lambda_k,1}  - \int_{\R^2_+} u_k^{p-1} \phi_k dx\Bigr)= 0.   
\end{align*}
Hence $w$ satisfies (\ref{eq:equation-w-limit-c-pos}) in this case. By (\ref{neweq-2}) and weak lower semicontinuity, this implies that 
\begin{align*}
\sup_{t \ge 0}&\Bigl( E_*(t w) + \frac{t^2 d^2}{2}  \int_{\R^2}|\partial_{x_2} w|^2dx \Bigr)= \Bigl(\frac{1}{2}-\frac{1}{p}\Bigr) \Bigl( \|w\|_{H^1(\R^2)}^2 +d^2 \int_{\R^2}|\partial_{x_2} w|^2 dx\Bigr) \\
&\le \Bigl(\frac{1}{2}-\frac{1}{p}\Bigr) \lim_{k \to \infty} \|u_k\|_{\lambda_k,1}^2= \lim_{k \to \infty} E_{\lambda_k}(u_k) = \lim_{k \to \infty} c_{\lambda_k,1} =  c_\infty.
\end{align*}
On the other hand, we have 
$$
c_\infty \le  \sup_{t \ge 0} E_*(t w) <  \sup_{t \ge 0}\Bigl( E_*(t w) + t^2 d^2 \int_{\R^2}|\partial_{x_2} w|^2dx \Bigr).
$$
Combining these inequalities yields a contradiction. Hence (\ref{eq:tau-k-infty-refined}) holds. 

\medskip
The same argument as above with $d=0$ yields that $w \geq 0$ is a solution of the limit problem 
$$
-\Delta w + w = w^{p-1} \qquad \text{in $\R^2$}
$$
and by uniqueness we have $w = w_\infty$ after adding a finite translation to the sequence $\tau_k$ if necessary. 

\medskip
We finish the proof by showing that $w_k \to w$ strongly in $H^1(\R^2)$. Indeed, by weak lower semicontinuity,
\begin{align*}
c_\infty &= \Bigl(\frac{1}{2}-\frac{1}{p}\Bigr) \|w\|_{H^1(\R^2)}^2 \le 
\Bigl(\frac{1}{2}-\frac{1}{p}\Bigr) \liminf_{k \to \infty} \|w_k\|_{H^1(\R^2)}^2\\
&= \Bigl(\frac{1}{2}-\frac{1}{p}\Bigr) \liminf_{k \to \infty} \|u_k\|_{H^1(\R^2_+)}^2 \le \Bigl(\frac{1}{2}-\frac{1}{p}\Bigr) \lim_{k \to \infty} \Bigl(\|u_k\|_{\lambda_k,1}^2\Bigr)\\
&= \lim_{k \to \infty}c_{\lambda_k}= c_\infty.
\end{align*}
Hence equality holds in all steps. Since $H^1(\R^2)$ is uniformly convex, this shows that $w_k \to w$ strongly in $H^1(\R^2)$, as claimed and this completes the proof of the theorem.
\end{proof}

\subsection{The limit $\lambda \to 0$}
\label{sec:scale-invar-case}
Next we consider the asymptotics of least energy solutions to (\ref{eq:untransformed-problem}) in the case $q=1$ as $\lambda \to 0$. To find a suitable limit problem, we consider the transformed Dirichlet problem 
\begin{equation} \label{transformed lambda to zero equation}
\left \{
  \begin{aligned}
-\Delta v - \del_\theta^2 v + \lambda^2 v &= |v|^{p-2} v &&\qquad \text{in $\R^2_+$,}\\
v &= 0&&\qquad \text{on $\partial \R^2_+$.}     
  \end{aligned}
\right.
\end{equation}
Weak solutions $v \in H^+$ of (\ref{transformed lambda to zero equation}) are critical points of the associated energy functional given by
$$
J_\lambda: H^+ \to \R, \quad J_\lambda(v)= \frac{1}{2} \left( | \nabla v|_2^2 + | \del_\theta v|_2^2 + \lambda^2 |v|_2^2 \right) - \frac{1}{p} \|v\|_p^p .
$$
These notions can be related to the original problem as follows: For $\lambda>0$, consider the transformation
$$
H^+ \ni u \mapsto v \in H^+, \quad v(x)=\lambda^\frac{2}{p-2} u(\lambda x)
$$
so that 
\begin{equation} \label{transformed energy relation}
	J_\lambda (v) = \lambda^\frac{4}{p-2} E_{\lambda}^+(u). 
	\end{equation}
Moreover, $u$ is a (least energy) solution of (\ref{eq:untransformed-problem}) if and only if $v$ is a (least energy) solution of (\ref{transformed lambda to zero equation}). 

In order to prove Theorem~\ref{rescaled convergence}, let $(\lambda_k)_k$ be sequence of numbers $\lambda_k \le 1$ such that $\lambda_k \to 0$ as $k \to \infty$ and let $u_k \in H^+$ be positive least energy solutions of \eqref{eq:untransformed-problem} for $\lambda=\lambda_k$.

\medskip
For any $k\in \N$, set
$$ 
v_k(x)= \lambda_k^{\frac{2}{p-2}}u_k(\lambda_k x), \quad v_k \in H^+.
$$

\begin{lemma} \label{boundedness}
  The sequence $(v_k)_k$ is bounded in $H^+$.
\end{lemma}
\begin{proof}
  By Corollary~\ref{H+ properties}, it suffices to show that there exists $C>0$ such that 
	$$
	\|v_k\|_{1,0} \leq C \qquad \text{for all $k \in \N$.}
	$$
By the remarks above, $v_k$ is a least energy solution of the transformed problem \eqref{transformed lambda to zero equation} with $\lambda=\lambda_k$. Multiplying this equation with $v_k$ and integrating by parts yields
	\begin{equation} \label{solution norm}
	\|v_k\|_{1,0}^2 +\lambda_k^2 |v_k|_2^2 = |v_k|_p^p \qquad \text{for all $k \in \N$.}
	\end{equation}
	Moreover, we have 
	$$
	J_{\lambda_k}(v_k) = \inf_{v \in H^+ \setminus \{0\}} \sup_{t \geq 0} J_{\lambda_k}(tv).
	$$
 Fix $\phi \in C_c^\infty(\R_+^2) \setminus \{ 0\}$. Since $v_{k}$ is a least energy solution of \eqref{transformed lambda to zero equation} for $\lambda=\lambda_k \le 1$, we have
$$
	J_{\lambda_k}(v_k) \leq \sup_{t \geq 0} J_{\lambda_k}(t \phi) \leq \sup_{t \geq 0} J_1(t \phi) =: C_0
	$$
	where, clearly, $C_0$ is independent of $k$. 
%
\medskip	
We can then use \eqref{solution norm} to get
	$$
	J_{\lambda_k}(v_k) = \left(\frac{1}{2} - \frac{1}{p} \right)\Bigl(\|v_k\|_{1,0}^2 +\lambda_k^2 |v_k|_2^2\Bigr) \geq \left(\frac{1}{2} - \frac{1}{p} \right) \|v_k\|_{1,0}^2 
	$$
	and hence
	$$
	\|v_k\|_{1,0}^2 \leq \frac{C_0}{\frac{1}{2} - \frac{1}{p}} \qquad \text{for all $k \in \N$.}
	$$
      \end{proof}
           
As a consequence of Lemma~\ref{boundedness}, we can pass to a subsequence and assume
	$$
	v_k \weakto v^*  \quad \text{in $H^+$.}
	$$
\begin{lemma}
  The weak limit $v^*$ is a nontrivial weak solution of (\ref{eq:untransformed-problem-limit-lambda-0}).
\end{lemma}

\begin{proof}
Since every $v_k$ is a weak solutions of \eqref{eq:untransformed-problem}, for any test function $\varphi\in C^{\infty}_c(\R^2_+)$ we have
$$
	\int_{\R^2_+} \left(  \nabla v_k \cdot \nabla \phi + \del_\theta v_k \del_\theta \phi \right) \, dx = \int_{\R^2_+} |v_k|^{p-2} v_k \phi \, dx - \lambda_k^2 \int_{\R^2_+} v_k \phi \, dx.
	$$
	
Besides, since $v_k \rightharpoonup v^*$ weakly in $H^+$ and $\lambda_k \to 0^+$ as $k \to \infty$,
$$
	\int_{\R^2_+} \left( \nabla v_k \cdot \nabla \phi + \del_\theta v_k \del_\theta \phi \right) \, dx - \lambda_k^2 \int_{\R^2_+}  v_k \phi \, dx\to \int_{\R^2_+} \left( \nabla v^* \cdot \nabla \phi + \del_\theta v^* \del_\theta \phi \right) \, dx,
	$$
	and 
	$$
	\int_{\R^2_+} |v_k|^{p-2} v_k \phi \, dx  \to \int_{\R^2+} |v^*|^{p-2} v^* \phi \, dx 
	$$
as a consequence of the compact embedding $H^+ \hookrightarrow L^p(\R^2_+)$. It then follows that $v^* \in H^+$ is a weak solution of 
	$$
	-\Delta v^* - \del_\theta^2 v^* = |v^*|^{p-2} v^* \quad \text{in $\R_+^2$.}
	$$
	
Next, we prove that $v^* \not \equiv 0$. To do so, first observe that the  embedding $H^+ \hookrightarrow L^p$ yields
	$$
	C:=\inf_{u \in H^+ \setminus \{0\}} \frac{\|u\|_{1,0}}{|u|_p}\in (0,\infty).
	$$
	
Thus, the above comments, together with the fact that $|u|_2^2 \leq | \del_\theta u |_2^2 \leq \|u\|_{1,0}^2$ for $u \in H^+$ (see Corollary~\ref{H+ properties}), imply that 
$$
C^2 =
	\inf_{u\in H^+\setminus \{0\}} \frac{\|u\|_{1,0}^{2}}{|u|_p^2} \leq
	\inf_{u\in H^+\setminus \{0\}} \frac{\|u\|_{1,0}^2 + \lambda_k^2 |u|_2^2}{|u|_p^2} \leq 2 \inf_{u\in H^+\setminus \{0\}} \frac{\|u\|_{1,0}^2}{|u|_p^2} = 2  C^2 .
	$$
Recalling also that 
$$
	J_{\lambda_k}(v_k) = \inf_{u\in H^+\setminus \{0\}} \left(\frac{1}{2} - \frac{1}{p} \right) \left( \frac{\|u\|_{1,0}^2 + \lambda_k^2 |u|_2^2}{|u|_p^2}\right)^\frac{p}{p-2},
	$$
we thus have
	\begin{equation} \label{upper-lower energy bounds}
	\left(\frac{1}{2}-\frac{1}{p} \right)C^\frac{2p}{p-2}  \leq J_{\lambda_k}(v_k) \leq \left(\frac{1}{2}-\frac{1}{p}\right)  \bigl(2 C^2\bigr)^\frac{p}{p-2} \qquad \text{for all $k \in \N$.}
      \end{equation}
      Now assume by contradiction that $v^*= 0$, i.e., $v_k \rightharpoonup 0$ weakly in $H^+$. The compact embedding $H^+ \hookrightarrow L^p$ implies $v_k \to 0$ in $L^p$, and therefore $\|v_k\|_{1,0} \to 0$ by \eqref{solution norm}. Hence also $|v_k|_2 \to 0$ by Corollary~\ref{H+ properties}. We then deduce that 
	$$
	J_{\lambda_k}(v_k) = \left(\frac{1}{2} - \frac{1}{p} \right) (\|v_k\|_{1,0}^2 + \lambda_k^2 |v_k|_2^2) \to 0,
	$$
which contradicts \eqref{upper-lower energy bounds}. We conclude that $v^* \not = 0$, as claimed.
\end{proof}

We will now use $\Gamma$-convergence to finish the proof of Theorem \ref{rescaled convergence}: 

\begin{proof}[Proof of Theorem \ref{rescaled convergence}]
  It remains to prove that $v^*$ is a least energy solution of
(\ref{eq:untransformed-problem-limit-lambda-0}), and that $v_k \to v^*$ strongly in $H^+$ as $k \to \infty$. 
\medskip

To deduce these properties from $\Gamma$-convergence theory, we consider the space $X:=H^+ \setminus \{ 0\}$ endowed with the weak topology (induced by $\|\cdot \|_{1,0}$). Consider the functionals $F_k,F:X \to [0,\infty]$ defined by
$$
\begin{aligned}
F_k(u)  := \frac{(\|u\|_{1,0}^2 + \lambda_k^2 |u|_2^2)^\frac{p}{p-2}}{|u|_p^\frac{2p}{p-2}} \quad \hbox{and}\quad F(u)  := \frac{\|u\|_{1,0}^\frac{2p}{p-2}}{|u|_p^\frac{2p}{p-2}} .
	\end{aligned}
$$
Then we have 
$$
F(u) \leq F_k(u) \qquad \text{for every $k\in \N$ and $u\in H^+$.}
$$
Let $(\tilde{u}_k)_k \subset X$ be an arbitrary sequence such that $\tilde{u}_k \to \tilde{u}$ in $X$ (recall that $X$ has the weak topology of $H^+$). The compact embedding $H^+\hookrightarrow L^p(\R^2_+)$ and the weak lower semicontinuity of $\|\cdot\|_{1,0}$ imply
$$
F(\tilde{u}) \leq \liminf_{k \to \infty} F(\tilde{u}_k) \leq \liminf_{k \to \infty} F_k(\tilde{u}_k) .
	$$
On the other hand, for any $\tilde{u} \in X$, the constant sequence $\tilde{u}_k :=\tilde{u}$ satisfies that $\tilde{u}_k \to \tilde{u}$ in $X$ and
$$
F(\tilde{u})=\lim \limits_{k \to \infty} F_k(\tilde{u}_k) .
$$
We conclude that $F_k\overset{\Gamma}{\to}F$. Since,
$$
F_k(v_k) = \inf_{u \in X} F_k(u)
$$
and $v_k \to  v$ in $X$, it follows from \cite[Corollary 7.20]{DalMaso1993Book} that 
\begin{equation}
  \label{eq:conclusion-gamma-convergence}
F(v)=\inf_{u \in X} F(u) = \lim \limits_{k \to \infty}F_k(v_k).
\end{equation}
Consequently,
$$
	\left(\frac{1}{2}-\frac{1}{p} \right) \frac{\|v\|_{1,0}^\frac{2p}{p-2}}{|v|_p^\frac{2p}{p-2}} = \inf_{u \in H^+ \setminus \{0\}} \left(\frac{1}{2}-\frac{1}{p} \right) \frac{\|u\|_{1,0}^\frac{2p}{p-2}}{|u|_p^\frac{2p}{p-2}} = \inf_{u \in H^+ \setminus \{0\}} \sup_{t \geq 0} \left( \frac{t^2}{2}\|u\|_{1,0}^2 - \frac{t^p}{p}  |u|_p^p \right) ,
	$$
	and this implies that $v$ is a least energy solution of (\ref{eq:untransformed-problem-limit-lambda-0}). Moreover, since $v_k \to v$ in $L^p(\R^2_+)$ by the compact embedding $H^+ \hookrightarrow L^p(\R^2_+)$, it follows from (\ref{eq:conclusion-gamma-convergence}) and the definition of the functionals $F_k$ and $F$ that   
        $$
        \|v\|_{1,0}^2  =  \lim_{k\to \infty}\Bigl(\|v_k\|_{1,0}^2 + \lambda_k^2 |v_k|_2^2\Bigr) \ge \limsup_{k \to \infty}\|v_k\|_{1,0}^2 \ge \liminf_{k \to \infty}\|v_k\|_{1,0}^2 \ge \|v\|_{1,0}^2.
        $$
        Consequently, we have
        $$
        \|v_k\|_{1,0} \to \|v\|_{1,0} \qquad \text{as $k \to \infty$,}
        $$
        and the uniform convexity of $(H^+,\|\cdot\|_{1,0})$ implies that 
 $v_k \to v$ strongly in $H^+$ as $k \to \infty$.
 \end{proof}

\section{Radial versus nonradial least energy nodal solutions}
\label{sec:radi-vers-nonr}
In this section we complete the proofs of Theorem~\ref{main theorem} and Theorem~\ref{non-radiality-L-infty-est-intro}. Given the assumptions of Theorem~\ref{main theorem}, the existence of a least energy nodal solution of (\ref{eq:untransformed-problem whole space}) for every $\lambda>0$ is a direct consequence of Corollary \ref{corollary1}.

We will now first prove Theorem~\ref{main theorem}(ii), which will be a consequence of Lemma~\ref{energy-asymptotics} and a result in \cite{Weth}. 

\medskip
We recall that, as in Section \ref{sec:limit-lambda-to-infty} and Section~\ref{sec:funct-analyt-fram}, the energy functionals $E_*,E_{\lambda}: H \to \R$ are defined by
$$
E_*(v):= \frac{1}{2}\int_{\R^2}\Bigl(|\nabla v|^2 + |v|^2\Bigr)dx -\frac{1}{p}\int_{\R^2}|v|^pdx
$$
and
$$
E_{\lambda}(v)=E_*(v) + \frac{1}{\lambda^2}\int_{\R^2}|\del_\theta v|^2dx
$$
for $v\in H$. Moreover, as in Section~\ref{sec:funct-analyt-fram}, we consider the $\lambda$-dependent scalar product $\langle \cdot, \cdot \rangle_\lambda$ defined in \eqref{eq:lambda-product} on $H$ and the corresponding norm $\|\cdot\|_\lambda$. In particular, we shall use $\|\cdot\|_1$ given by
$$
\|u\|_1^2 = \int_{\R^2}\bigl(|\nabla u|^2+ |\del_\theta u|^2 + |u|^2\bigr)\,dx \qquad \text{for $u \in H$.}
$$

\medskip

\begin{proposition}
\label{energy-doubling-plus-eps}
There exists $\eps_*>0$ such that for every $\lambda>0$ and every radial nodal solution $u \in H$ of (\ref{eq:untransformed-problem whole space}) we have
 $$
E_*(u) = E_\lambda(u) > 2c_\infty + \eps_*,
$$
where $c_\infty$ is given in (\ref{eq:def-c-limits}).  
\end{proposition}

\begin{proof}
First observe that $E_*(u) = E_\lambda(u)$ for every radial function $u \in H$. Moreover, if $u$ is a radial nodal solution of (\ref{eq:untransformed-problem whole space}), then $u$ also solves the limit problem (\ref{eq:unique-rad-pos-solution-limit-problem}). By \cite[Theorem 1.5]{Weth}, and the variational characterization of $c_\infty$ given by \eqref{eq:def-limit ground state} and \eqref{eq:limit-c-infty}, there exists $\eps_*>0$ with the property that $E_*(u) > 2c_\infty + \eps_*$ for every nodal solution of (\ref{eq:unique-rad-pos-solution-limit-problem}). This proves the claim.   
\end{proof}

\begin{proof}[Proof of Theorem~\ref{main theorem}(ii) (completed)]
  Let $\eps_*>0$ be given by Proposition~\ref{energy-doubling-plus-eps}. By \eqref{eq:def-c-limits}, there exists $\Lambda_0>0$ with the property that
  $$
  c_{\lambda}< c_\infty + \frac{\eps_*}{2} \qquad \text{for every $\lambda> \Lambda_0$.}
  $$
Consequently, for $\lambda>\Lambda_0$, problem (\ref{eq:untransformed-problem}) admits a nontrivial solution $u \in H^+$ with $E_\lambda^+(u)< c_\infty + \frac{\eps_*}{2}$. By odd reflection, we may extend $u$ to a nodal solution of (\ref{eq:untransformed-problem whole space}) with $E_\lambda(u) < 2c_\infty + \eps_*$. Proposition \ref{energy-doubling-plus-eps} therefore implies that the least energy nodal solutions of (\ref{eq:untransformed-problem whole space}) cannot be radial. 
\end{proof}

\medskip
Next, we complete the proof of Theorem~\ref{non-radiality-L-infty-est-intro}, which we restate here for the reader's convenience.

\begin{theorem} \label{non-radiality-L-infty-est}
  Let $p >2$.
  \begin{enumerate}
  \item If $u \in H$ is a nontrivial weak solution of
\begin{equation}
  \label{eq:equation-whole-space-section-radiality}
	-\Delta u - \frac{1}{\lambda^2} \del_\theta^2 u + u = |u|^{p-2} u \quad \text{in $\R^2$}
\end{equation}
for some $\lambda>0$ satisfying $\lambda < \Bigl(\frac{1}{(p-1)|u|_\infty^{p-2}}\Bigr)^{\frac{1}{2}}$, then $u$ is a radial function. 
  \item For every $c>0$, there exists $\lambda_c>0$ with the property that every weak solution $u \in H$ of (\ref{eq:equation-whole-space-section-radiality})
for some $\lambda \in (0,\lambda_c)$ with $E_\lambda(u) \le c$ is radial.
\end{enumerate}
\end{theorem}

\begin{proof}
(i) Let $u \in H$ be a nontrivial weak solution of (\ref{eq:equation-whole-space-section-radiality}) for some $\lambda>0$, and let, as before, $u^\#$ denote the radial average of $u$ as defined in (\ref{eq:radial-averaging}). It is easy to see that, for every $k \in \N$, the function $u^\# \in H$ is a weak solution of 
$$
-\Delta u^\#  + u^\# = \left(|u|^{p-2} u \right)^\# \quad \text{in $\R^2$}.
$$
Consequently we have, in weak sense, 
\begin{equation*}
-\Delta (u-u^\#) - \frac{1}{\lambda^2} \del_\theta^2 (u-u^\#)  + (u- u^\#) = |u|^{p-2} u - \left(|u|^{p-2} u \right)^\# \quad \text{in $\R^2$.}
\end{equation*}
Testing this equation against $u-u^\#$ yields
\begin{align} 
  &\frac{1}{\lambda^2} |\del_\theta u |_2^2= \frac{1}{\lambda^2} |\del_\theta (u-u^\#) |_2^2 \le  |\nabla(u-u^\#)|_2^2 + \frac{1}{\lambda^2} |\del_\theta (u-u^\#) |_2^2 + |u-u^\#|_2^2 \nonumber \\
  & = \int_{\R^2} \left( |u|^{p-2} u - \left(|u|^{p-2} u \right)^\# \right) (u-u^\#) \, dx  \leq   \left| |u|^{p-2} u - \left(|u|^{p-2} u \right)^\# \right|_2 \, |u-u^\#|_2  \nonumber \\
  &\leq   \left| |u|^{p-2} u - \left(|u|^{p-2} u \right)^\# \right|_2 \, |\del_\theta u|_2  \label{eq:comparH1uurad},
\end{align}
where we used Lemma~\ref{lemma1} in the last step. Moreover, $|u|^{p-2} u \in H$ by Remark~\ref{remark-l-infty-H-functions}, and therefore
\begin{equation}
  \left| |u|^{p-2} u - \left(|u|^{p-2} u \right)^\# \right|_2  \le \bigl|\del_\theta \bigl(|u|^{p-2} u\bigr)\bigr|_2 = (p-1) \bigl| |u|^{p-2} \del_\theta  u \bigr|_2 \le (p-1)|u|_\infty^{p-2}|\del_\theta  u |_2  \label{eq:comparH1uurad-1},
\end{equation}
again by Lemma~\ref{lemma1}. Combining (\ref{eq:comparH1uurad}) and (\ref{eq:comparH1uurad-1}), we obtain that
$$
\frac{1}{\lambda^2} |\del_\theta u |_2^2 \le (p-1)|u|_\infty^{p-2 } |\del_\theta  u |_2^2
$$
which implies that $\del_\theta u \equiv 0$ if $\lambda < \Bigl(\frac{1}{(p-1)|u|_\infty^{p-2}}\Bigr)^{\frac{1}{2}}$. The proof of (i) is thus finished.

(ii) Let $c>0$ be given, and let $u \in H$ be a nontrivial weak solution of (\ref{eq:equation-whole-space-section-radiality}) for some $\lambda>0$ with $E_\lambda (u) \le c$. Since $E_\lambda (u) =\Bigl(\frac{1}{2}-\frac{1}{p}\Bigr)\|u\|_\lambda^2$, it then follows that
  \begin{equation*}
\|u\|_{H^1(\R^2)}^2 \le \|u\|_\lambda^2  = \frac{2p}{p-2}E_\lambda (u) \leq  \frac{2p c}{p-2} 
\end{equation*}
and therefore
$$
|u|_\infty \le C \|u\|_{H^1(\R^2)}^\sigma \le C \Bigl(\frac{2p c}{p-2} \Bigr)^{\frac{\sigma}{2}} =:\mu_c
$$
by Lemma~\ref{L-infty regularity} with the constants $C,\sigma>0$ given there. Hence, if
$$
\lambda < \lambda_c:= \Bigl(\frac{1}{(p-1)\mu_c^{p-2}}\Bigr)^{\frac{1}{2}},
$$
then also $\lambda< \Bigl(\frac{1}{(p-1)|u|_\infty^{p-2}}\Bigr)^{\frac{1}{2}}$ and therefore $u$ is radial by (i). The proof is finished.
\end{proof}

Next we provide uniform energy estimates for least energy nodal solutions of (\ref{eq:equation-whole-space-section-radiality}).
            
\begin{lemma} \label{Upper lower energy bound}
Let $p>2$.  There exist constants $c,C>0$ with the property that
  \begin{equation}
    \label{eq:upper-lower-energy-bound-equation}
  c \leq E_\lambda(u) \leq C  
  \end{equation}
  for every $\lambda>0$ and every least energy nodal solution $u \in H$ of (\ref{eq:equation-whole-space-section-radiality}). 
\end{lemma}
    
\begin{proof}
  The lower bound is obtained by choosing $c = \hat c_\infty$ as defined in (\ref{eq:def-limit ground state}), since
  $$
  E_\lambda(u) = \sup_{t \ge 0} E_\lambda(tu) \ge \sup_{t \ge 0} E_*(tu) \ge
\hat c_\infty
$$
for every $\lambda>0$ and every nontrivial solution $u \in H$ of (\ref{eq:equation-whole-space-section-radiality}). 

\medskip
For the upper bound, we first remark that  the existence of radial nodal solutions of (\ref{eq:unique-rad-pos-solution-limit-problem}) is well known, see for instance Theorems 4 and 5 in \cite{Strauss}. Let $\hat{u} \in H^1(\R^2)$ be a fixed radial nodal solution of (\ref{eq:unique-rad-pos-solution-limit-problem}) and set $C= E_*(\hat{u})$. For every $\lambda>0$, the function $\hat{u} \in H$ is then also a nodal solution of (\ref{eq:equation-whole-space-section-radiality}), and therefore 
  $$
  E_\lambda(u) \le E_\lambda(\hat{u}) = E_*(\hat{u})= C
  $$
for every least energy nodal solution $u \in H$ of (\ref{eq:equation-whole-space-section-radiality}). 
\end{proof}
\medskip

{\em The proof of Theorem~\ref{main theorem} is now completed by deriving Part (i) of this theorem as follows.}
\vspace{-0.2cm}

Let $C>0$ be given by Lemma~\ref{Upper lower energy bound}, and let $u \in H$ be a least energy solution of (\ref{eq:equation-whole-space-section-radiality}) for some $\lambda>0$. Then we have $E_\lambda(u) \le
C$. Applying Theorem~\ref{non-radiality-L-infty-est} with $c=C$ and considering $\lambda_0:= \min \{\lambda_c, \Lambda_0\}$ with $\Lambda_0>0$ given as in Theorem~\ref{main theorem}(ii), we then deduce that $0< \lambda_0 \le \Lambda_0$, and $u$ is radial if $\lambda < \lambda_0$. The proof of Theorem~\ref{main theorem}(i) is thus finished.

\appendix

\section{}

We give the proof of Lemma \ref{L-infty regularity}, which we restate here for the reader's convenience.

\begin{lemma} \label{app: regularity}
	Let $\lambda>0$ and let $u \in H$ be a weak solution of 
	\begin{equation} \label{app: equation}
	-\Delta u - \frac{1}{\lambda^2} \del_\theta^2 u + u = |u|^{p-2} u \quad \text{in $\R^2$}.
	\end{equation}
	Then $u \in L^\infty(\R^2)$. Furthermore, there exist constants $C,\sigma>0$, depending on $p>2$ but not on $u$ and $\lambda$, such that
	\begin{equation} \label{Linfty bound}
	|u|_\infty \leq C \|u\|_{H^1(\R^2)}^\sigma .
	\end{equation}
\end{lemma}
\begin{proof}
	The proof is based on Moser iteration, cf. Appendix B in \cite{Struwe} and the references therein. 
	
	We fix $L,s \geq 2$ and consider auxiliary functions $h, g \in C^1([0,\infty))$ defined by 
$$
	h(t) := s \int_0^t \min \{\tau^{s-1},L^{s-1}\}\,d\tau \qquad \text{and}\qquad g(t) := \int_0^t [h'(\tau)]^2 \, d\tau 
$$
We note that 
\begin{equation}
  \label{eq:additional-appendix}
h(t)= t^s\quad \text{for $t \le L$}\qquad \text{and}\qquad g(t) \leq t g'(t)= t(h'(t))^2 \quad \text{for $t \geq 0$,}  
\end{equation}
 since the function $t \mapsto h'(t) = s \min \{t^{s-1},L^{s-1}\}$ is nondecreasing.  We shall now show that $w := u^+ \in L^\infty(\R^2)$, and that $\|w\|_\infty$ is bounded by the r.h.s. of \eqref{Linfty bound}. Since we may replace $u$ with $-u$, the claim will then follow.
 
\medskip
We note that $w \in H$ and $\phi:= g(w) \in H$ with 
$$
\nabla w = \mathbf{1}_{\{u>0\}} \nabla u, \qquad \nabla \phi= g'(w) \nabla w , \qquad \del_\theta w  = \mathbf{1}_{\{u>0\}} \del_\theta u, \qquad \del_\theta \phi= g'(w) \del_\theta w .
$$

\medskip
This follows from the boundedness of $g'$ and the estimate $g(t) \leq s^2 t^{2s-1}$ for $t \ge 0$. Testing \eqref{app: equation} with $\phi$ gives
	$$
	\int_{\R^2} \Bigl(\nabla u \cdot \nabla \phi + \frac{1}{\lambda^2} (\del_\theta u \ \del_\theta \phi) + u \phi\Bigr)dx = \int_{\R^2} |u|^{p-2} u \phi  \, dx ,
	$$
from where we estimate,	
\begin{align}
\int_{\R^2} \Bigl(|\nabla h(w) |^2 +  \frac{1}{\lambda^2}(\del_\theta h(w))^2 + w g(w)\Bigr) dx &=\int_{\R^2} \Bigl(g'(w) \Bigl(|\nabla w |^2 +  \frac{1}{\lambda^2}(\del_\theta w)^2 \Bigr) + u g(w)\Bigr) dx \nonumber\\
&=  \int_{\R^2} |u|^{p-2} u g(w) \, dx \nonumber\\
& \leq \int_{\R^2} w^p  (h'(w))^2  \, dx.\label{appendix-first-inequality}
\end{align}

\medskip
Here we used (\ref{eq:additional-appendix}) in the last step. We now fix $r >1$ with $\frac{(p-2)r}{r-1} \geq 2$ and $q>4r$. Combining (\ref{appendix-first-inequality}) with Sobolev embeddings, we obtain the inequality 
\begin{equation}
  \label{eq:appendix-second-inequality}
	\frac{1}{c_0} |h(w)|_q^2 - |h(w)|_2^2    + \int_{\R^2}  w g(w) \, dx \leq  \int_{\R^2} w^p  (h'(w))^2  \, dx 
\end{equation}
with a constant $c_0= c_0(q)>0$. Since 
$$
h(t)= t^s,\quad h'(t) = st^{s-1}\quad \text{and}\quad g(t)= s^2 \int_0^t \tau^{2s-2}\,d\tau = \frac{s^2}{2s-1} t^{2s-1}\qquad \text{for $t \le L$,}
$$
we may let $L \to \infty$ in (\ref{eq:appendix-second-inequality}) and apply Lebesgue's theorem to obtain 
$$
	\frac{1}{c_0}  |w^s|_q^2 +\Bigl(\frac{s^2}{2s-1}-1\Bigr)|w^s|_2^2 \leq  s^2 \int_{\R^2} w^{p+2s-2} \, dx \le s^2 |w|_{ \frac{(p-2)r}{r-1}}^{p-2}|w|_{2rs}^{2s}
$$
Since $s \ge 2$, we have $\frac{s^2}{2s-1}\ge 1$, and we thus obtain the inequality 
\begin{equation}
  \label{eq:appendix-third-inequality}
|w|_{sq} \le (c_1s)^{\frac{1}{s}}  |w|_{2rs}\qquad \text{with $c_1:= \Bigl(c_0 
|w|_{ \frac{r(p-2)}{r-1}}^{p-2}\Bigr)^{\frac{1}{2}}$.} 
\end{equation}
Next we note that the choice of $r$ and $q$ only depends on $p$ but not on $s\ge 2$. We may therefore consider $s=s_n= \rho^n$ for $n \in \N$ with $\rho := \frac{q}{2r} >2$, so that 
	$$
	2s_1 r = q \qquad \text{and}\qquad 2s_{n+1} r = q s_n  \quad \text{for $n \in \N$.}
	$$
 Iteration of (\ref{eq:appendix-third-inequality}) then gives 
	$$
	|w|_{\rho^n q} = |w|_{s_n q} \leq |w|_q \prod_{j=1}^n (c_1 \rho^j)^{\rho^{-j}}  \leq c_1^{\frac{\rho}{\rho-1}} c_2 |w|_q  \quad \text{for all $n$}\quad \text{with}\; c_2 : =  \rho^{\sum_{j=1}^\infty j \rho^{-j}} < \infty.
	$$
	It follows that
	\begin{equation} \label{p-norm limit}
	|w|_\infty = \lim_{n \to \infty} |w|_{ \rho^n q} \leq c_1^{\frac{\rho}{\rho-1}} c_2 |w|_q . 
	\end{equation}
	Moreover, by Sobolev embeddings, we have 
	$$
	c_1 \leq c_1' \|w\|_{H^1(\R^2)}^{\frac{p-2}{2}} \le c_1' \|u\|_{H^1(\R^2)}^{\frac{p-2}{2}}\qquad \text{and}\qquad  |w|_q\le \tilde c  \|w\|_{H^1} \le \tilde c  \|u\|_{H^1(\R^2)}  
	$$
        with constants $c_1', \tilde c>0$ depending only on $p,r$ and $q$. It thus follows from \eqref{p-norm limit} that
	$$
	|w|_\infty \leq C \|u\|_{H^1(\R^2)}^{\frac{(p-2)\rho}{2(\rho-1)} +1}\qquad \text{with}\qquad C:= c_2 (c_1')^{\frac{\rho}{\rho-1}}\tilde c. 
	$$
The proof is thus finished.	
\end{proof}

\begin{remark}
  \label{Appendix-remark}
{\rm  Let $\lambda>0$ and $p \in (2,\infty)$. By a variant of the Moser iteration argument given above, we can also show that every weak solution $u \in H^+$ of 
	\begin{equation} \label{app: equation-half-space}
	-\Delta u - \frac{1}{\lambda^2} \del_\theta^2 u = |u|^{p-2} u \quad \text{in $\R^2_+$,}\qquad u = 0 \quad \text{on $\partial \R^2_+$}
      \end{equation}
      satisfies $u \in L^\infty(\R^2_+)$. To see this, we replace, with the help of Corollary~\ref{H+ properties} and (\ref{app: equation-half-space}), the inequalities (\ref{appendix-first-inequality}) and (\ref{eq:appendix-second-inequality}) by 
      $$
      \frac{1}{c}|h(w)|^2_q \le \|h(w)\|_{\lambda,0}^2 = \int_{\R^2_+} |u|^{p-2}u g(w)\,dx \le \int_{\R^2_+} w^p \bigl(h'(w)\bigr)^2\,dx
      $$
with a constant $c>0$ depending on $q$ and $\lambda$. We can then complete the argument as above, noting that in this case the constants depend on $\lambda>0$.} 
\end{remark}

\end{document}